\documentclass{amsart}
\usepackage{amsfonts}

\newtheorem{theorem}{Theorem}
\theoremstyle{plain}

\newtheorem{corollary}{Corollary}

\newtheorem{definition}{Definition}
\newtheorem{example}{Example}

\newtheorem{lemma}{Lemma}

\newtheorem{remark}{Remark}

\numberwithin{equation}{section}
\input{tcilatex}

\begin{document}
\title[The Relaxed SMP in the Mean-field Singular Controls]{The Relaxed
Stochastic Maximum Principle in the Mean-field Singular Controls}
\author{Liangquan Zhang}
\address{1. School of Mathematics, Shandong University, Jinan 250100,
People's Republic of China. {2. }Laboratoire de Math\'{e}matiques, Universit%
\'{e} de Bretagne Occidentale, 29285 Brest C\'{e}dex, France. }
\email{xiaoquan51011@163.com}
\urladdr{}
\thanks{This work was partially supported by Marie Curie Initial Training
Network (ITN) project: \textquotedblright Deterministic and Stochastic
Controlled System and Application\textquotedblright ,
FP7-PEOPLE-2007-1-1-ITN, No. 213841-2. }
\date{01/11/2012}
\dedicatory{}
\thanks{}

\begin{abstract}
In this paper, we study the optimal control system driven by stochastic
differential equations (SDEs) of mean-field type, in which the control
variable has two components, the first being absolutely continuous and the
second singular. On the other hand, the coefficients depend on the state of
the solution process as well as of its expected value. Moreover, the cost
functional is also of mean field type. This makes the control problem time
inconsistent in the sense that the Bellman optimality principle does not
hold. Our aim is to derive a stochastic maximum principle of optimal control
of Pontriagin type to the class of measure-valued controls.
\end{abstract}

\maketitle

\section{Introduction}

Let $v=\left( \Omega ,\mathcal{F},\left( \mathcal{F}_t\right) _{t\geq
0},P,W\right) $ be a reference probability system composed of a completed
probability space $\left( \Omega ,\mathcal{F},P\right) ,$ a filtration $%
\left( \mathcal{F}_t\right) _{t\geq 0}$ satisfying the usual assumptions of
right-continuity and completeness, and a $d$-dimensional $\left( \mathcal{F}%
_t\right) $-Brownian motion $W$ defined on $\left( \Omega ,\mathcal{F}%
,P\right) .$

Consider the following mean-field controlled stochastic differential
equations: 
\begin{equation}
\left\{ 
\begin{array}{lll}
\text{d}X\left( t\right) & = & b\left( t,X\left( t\right) ,\mathbb{E}\left[
X\left( t\right) \right] ,u\left( t\right) \right) \text{d}t \\ 
&  & +\sigma \left( t,X\left( t\right) ,\mathbb{E}\left[ X\left( t\right) %
\right] \right) \text{d}W\left( t\right) +G\left( t\right) \text{d}\eta
\left( t\right) , \\ 
X\left( 0\right) & = & x_{0}\in \mathbb{R}^{n},\quad t\in \left[ 0,T\right] .%
\end{array}%
\right.  \tag{1.1}
\end{equation}%
The coefficients $b,\sigma $ and $G$ will be defined below and $W$ is the
Borwnian motion. For every $t$, the control $u\left( t\right) $ ($\eta
\left( t\right) $) is allowed to take values in some control state space $U$
($\left( \left[ 0,+\infty \right) \right) ^{m}$). This mean-field SDEs is
obtained as the mean-square limit, when $n\rightarrow +\infty $, of a system
of interacting particles 
\begin{eqnarray*}
\text{d}X^{i,n}\left( t\right) &=&b\left( t,X^{i,n}\left( t\right) ,\frac{1}{%
n}\sum_{j=1}^{n}X^{j,n}\left( t\right) ,u\left( t\right) \right) \text{d}t \\
&&+\sigma \left( t,X^{i,n}\left( t\right) ,\frac{1}{n}\sum_{j=1}^{n}X^{j,n}%
\left( t\right) \right) \text{d}W^{i}\left( t\right) +G\left( t\right) \text{%
d}\eta \left( t\right) .
\end{eqnarray*}%
The classical example is the McKean-Vlasov model (see e.g. [26] and the
references therein).

The object of the control problem is to minimize a criteria, over the set $%
U\times \left( \left[ 0,+\infty \right) \right) ^{m}$, has the following form%
\begin{eqnarray*}
&&J\left( \left( u\left( \cdot \right) ,\eta \left( \cdot \right) \right)
\right) \\
&=&\mathbb{E}\left[ \int_{0}^{T}f\left( t,X^{u,\eta }\left( t\right) ,%
\mathbb{E}\left[ X^{u,\eta }\left( t\right) \right] ,u\left( t\right)
\right) \text{d}t\right. \\
&&\left. +h\left( X^{u,\eta }\left( T\right) ,\mathbb{E}\left[ X^{u,\eta
}\left( T\right) \right] \right) +\int_{0}^{T}\varphi \left( t\right) \text{d%
}\eta \left( t\right) \right] .
\end{eqnarray*}

The fundamental work on the stochastic maximum principle was obtained by
Kushner [24]. Since then there have been a lot of literature on this
subject, among them, in particular, those by Bensoussan [4], Bismut [5]
references therein.

The fact that the cost functional $J$ may be nonlinear with respect to the
expectation, makes the control problem time inconsistent in the sense that
Bellman's optimality principle, based on applying the law of iterated
conditional expectations on the cost functional, does not hold. A way to
solve this control problem is to device an extended version of the Dynamic
Programming Principle, as suggested in Ahmed and Ding [1]. The other result
in this direction was obtained independently by Li [27] and Andersson and
Djehiche [3], under the condition that the action space $U$ is convex.
Besides, in Meyer-Brandis, \O sendal, Zhou [28] a stochastic maximum
principle of mean-field type in a similar setting is studied by virtue of
Malliavin calculus. For nonconvex control domain, Buckdahn, et al, in [12]
obtained the Peng's maximum principle with two adjoint equations.

On the other hand, singular control problems have been studied by many
authors including Ben\u{e}s, Shepp, and Witsenhausen [6], Chow, Menaldi, and
Robin [13], Karatzas and Shreve [25] (for more information see references
therein). The approaches used there are mainly based on dynamic programming
principle. It was shown in particular that the value function is a solution
of a variational inequality, and the optimal state is a reflected diffusion
at the free boundary.

As we have known that stochastic maximum principle (SMP in short) is one way
to derive necessary conditions for some optimal controls. The first version
SMP for singular control problems was obtained by Cadenillas and Haussmannn
[14], and developed further by Bahlali et al [8], [9], Andersson, [2].
Recently, the version of stochastic maximum principle for relaxed-singular
controls was established by Bahlali, Djehiche and Mezerdi [8] in the case of
uncontrolled diffusion. In their paper, they first proved a first order
stochastic maximum principle for strict controls by using spike variation of
the absolutely continuous part of the control and a convex perturbation of
the singular part. Then by applying Ekeland's variational principle, they
established necessary conditions for near optimality, satisfied by a
sequence of strict controls converging in some sense to the relaxed optimal
control, by the so called chattering lemma. The relaxed maximum principle is
then derived by using some stability properties of the trajectories and the
adjoint processes with respect to the control variable. For diffusion term
containing control variable see [2]. Note that under the frame work of
mean-field which is time inconsistent in the sense that the Bellman
optimality principle does not hold. Hence, we adopt the approach developed
in [8] to deal with mean-field type.

The rest of this paper is organized as follows. After the statement of the
problem in the second section, we devote the third section to developing the
study the strict-singular control problems. In the last section, we will
establish necessary conditions of optimality for relaxed-singular controls.

\section{Notations and Statement of the Problem}

Let $T>0$ be a fixed time horizon and $\left( \Omega ,\mathcal{F},P\right) $
be a given filtered probability space on which a $d$-dimensional standard
Brownian motion $W=\{W\left( s\right) \},$ $s\geq 0$ is given, and the
filtration $\mathbb{F}=\{\mathcal{F}_{s},0\leq s\leq T\}$ is the natural
filtration of $W$ augmented by $P$-null sets of $\mathcal{F}$.

Let $U_{1}$ be a nonempty compact subset of $\mathbb{R}^{k}$ and $%
U_{2}=\left( \left[ 0,+\infty \right) \right) ^{m}.$ An admissible control $%
u $ is an $\mathbb{F}$-adapted and square-integrable process with values in $%
U_{1}$. We denote the set of all admissible controls by $\mathcal{U}_{1}$.
Besides, we denote $\mathcal{U}_{2}$ as the class of measurable, adapted
processes $\eta $ such that $\eta $ is bounded variation, nondecreasing,
left-continuous with right limits, $\eta \left( 0\right) =0$ and $\mathbb{E}%
\left[ \left\vert \eta \left( T\right) \right\vert ^{2}\right] <+\infty .$

We consider the following stochastic control system:%
\begin{equation}
\left\{ 
\begin{array}{lll}
\text{d}X^{u,\eta }\left( t\right) & = & b\left( t,X^{u,\eta }\left(
t\right) ,\mathbb{E}\left[ X^{u,\eta }\left( t\right) \right] ,u\left(
t\right) \right) \text{d}t \\ 
&  & +\sigma \left( t,X^{u,\eta }\left( t\right) ,\mathbb{E}\left[ X^{u,\eta
}\left( t\right) \right] \right) \text{d}W\left( t\right) +G\left( t\right) 
\text{d}\eta \left( t\right) , \\ 
X^{u,\eta }\left( 0\right) & = & x_{0}\in \mathbb{R}^{n},\quad t\in \left[
0,+\infty \right) ,%
\end{array}%
\right.  \tag{2.0.1}
\end{equation}%
where%
\begin{eqnarray*}
b\left( t,x,y,u\right) &:&\left[ 0,T\right] \times \mathbb{R}^{n}\mathbb{%
\times R}^{n}\mathbb{\times R}^{k}\mathbb{\rightarrow R}^{n}\mathbf{,} \\
\sigma \left( t,x,y\right) &:&\left[ 0,T\right] \times \mathbb{R}^{n}\mathbb{%
\times R}^{n}\mathbb{\rightarrow R}^{n\times d}\mathbf{,} \\
G\left( x\right) &:&\left[ 0,T\right] \rightarrow \mathbb{R}^{n\times m}, \\
t &\in &\left[ 0,T\right] ,\text{ }x,\text{ }y\in \mathbb{R}^{n}\mathbf{,}%
\text{ }u\in U.
\end{eqnarray*}

\subsubsection{Classical Singular Optimal Control Model}

The optimal control problem we are concerned with is to minimize the
following cost functional over $\mathcal{U}_{1}\times \mathcal{U}_{2}$ 
\begin{eqnarray*}
&&J\left( \left( \bar{u}\left( \cdot \right) ,\bar{\eta}\left( \cdot \right)
\right) \right) \\
&=&\mathbb{E}\left[ \int_{0}^{T}f\left( t,X^{\bar{u},\bar{\eta}}\left(
t\right) ,\mathbb{E}\left[ X^{\bar{u},\bar{\eta}}\left( t\right) \right]
,u\left( t\right) \right) \text{d}t\right. \\
&&\left. +h\left( X^{\bar{u},\bar{\eta}}\left( T\right) ,\mathbb{E}\left[ X^{%
\bar{u},\bar{\eta}}\left( T\right) \right] \right) +\int_{0}^{T}\varphi
\left( t\right) \text{d}\eta \left( t\right) \right] ,
\end{eqnarray*}%
\begin{equation}
\tag{2.1.1}
\end{equation}%
where%
\begin{eqnarray*}
f\left( t,x,y,u\right) &:&\left[ 0,T\right] \times \mathbb{R}^{n}\mathbb{%
\times R}^{n}\mathbb{\times R}^{k}\mathbb{\rightarrow R}\mathbf{,} \\
h\left( x,y\right) &:&\mathbb{R}^{n}\mathbb{\times R}^{n}\mathbb{\rightarrow
R}\mathbf{,} \\
\varphi \left( t\right) &:&\left[ 0,T\right] \rightarrow \mathbb{R}^{m}, \\
t &\in &\left[ 0,T\right] ,\text{ }x,\text{ }y\in \mathbb{R}^{n}\mathbf{,}%
\text{ }u\in U.
\end{eqnarray*}%
Any $\left( u\left( \cdot \right) ,\eta \left( \cdot \right) \right) \in 
\mathcal{U}_{1}\times \mathcal{U}_{2}$ satisfying%
\begin{equation}
J\left( \left( u\left( \cdot \right) ,\eta \left( \cdot \right) \right)
\right) =\inf\limits_{\left( \bar{u}\left( \cdot \right) ,\bar{\eta}\left(
\cdot \right) \right) \in \mathcal{U}_{1}\times \mathcal{U}_{2}}J\left( \bar{%
u}\left( \cdot \right) ,\bar{\eta}\left( \cdot \right) \right)  \tag{2.1.2}
\end{equation}%
is called a pair of singular optimal control. The corresponding state
process, solution of (2.0.1), is denoted by $X^{u\left( \cdot \right) ,\eta
\left( \cdot \right) }\left( \cdot \right) .$

We assume that

\begin{enumerate}
\item[\textbf{(H1)}] Assume that functions $b,f,h$ $\sigma $ are
continuously differentiable with respect $\left( x,y\right) $. Moreover,
They and their derivatives are continuous in $\left( x,y,u\right) $ and
bounded uniformly in $u$.

\item[\textbf{(H2)}] $b$ and $\sigma $ are bounded by $C\left( 1+\left\vert
x\right\vert +\left\vert y\right\vert +\left\vert u\right\vert \right) $ and 
$C\left( 1+\left\vert x\right\vert +\left\vert y\right\vert \right) ,$
respectively.

\item[\textbf{(H3)}] $G$ and $k$ are continuous and $G$ is bounded.
\end{enumerate}

Under the above hypothesis, Eq. (2.0.1) has a unique strong solution.

\section{Strictly Singular Optimal Control Problem}

\subsection{The maximum principle for strict controls}

At the beginning let us suppose that $\left( \left( u\left( \cdot \right)
,\eta \left( \cdot \right) \right) \right) $ is an optimal strict control
and denote by $X^{u,\eta }\left( \cdot \right) $ the optimal solution of
(2.0.1). The strict maximum principle will be proved in two steps. The first
variational inequality is derived from the fact 
\begin{equation}
J\left( u^{\varepsilon }\left( \cdot \right) ,\eta \left( \cdot \right)
\right) -J\left( u\left( \cdot \right) ,\eta \left( \cdot \right) \right)
\geq 0  \tag{3.1.1}
\end{equation}%
where $u^{\varepsilon }\left( \cdot \right) $ is a spike variation of the
absolutely continuous part on a small time interval. The second variational
inequity is attained from the inequity 
\begin{equation}
J\left( u\left( \cdot \right) ,\eta ^{\varepsilon }\left( \cdot \right)
\right) -J\left( u\left( \cdot \right) ,\eta \left( \cdot \right) \right)
\geq 0  \tag{3.1.2}
\end{equation}%
where $\eta ^{\varepsilon }\left( \cdot \right) $ is a convex perturbation
of $\eta .$

We consider the first variational inequality. Suppose $X^{u,\eta }\left(
t\right) $ is the solution to our optimal control problem. We introduce the
following spike variational control%
\begin{equation}
u^{\varepsilon }\left( t\right) =\left\{ 
\begin{array}{ll}
v, & \tau \leq t\leq \tau +\varepsilon , \\ 
u\left( t\right) , & \text{otherwise,}%
\end{array}%
\right.  \tag{3.1.3}
\end{equation}%
where $\varepsilon >0$ is sufficiently small, $\tau \in \left[ 0,T\right] .$ 
$v$ is an arbitrary $\mathcal{F}_{\tau }$-measurable random variable with
values in compact $U,$ $0\leq t\leq T,$ and $\sup\limits_{\omega \in \Omega
}\left\vert v\left( \omega \right) \right\vert <+\infty .$ Let $%
X^{u^{\varepsilon },\eta }\left( t\right) $ be the trajectory of the control
system (2.0.1) corresponding to the control $u^{\varepsilon }\left( t\right)
.$

We introduce the following variational equations%
\begin{equation}
\left\{ 
\begin{array}{l}
\text{d}y^{1}\left( t\right) =\left[ b_{x}\left( t,X^{u,\eta }\left(
t\right) ,\mathbb{E}\left[ X^{u,\eta }\left( t\right) \right] ,u\left(
t\right) \right) y^{1}\left( t\right) \right. \\ 
\qquad +b_{\tilde{x}}\left( t,X^{u,\eta }\left( t\right) ,\mathbb{E}\left[
X^{u,\eta }\left( t\right) \right] ,u\left( t\right) \right) \mathbb{E}\left[
y^{1}\left( t\right) \right] \\ 
\qquad \left. +b\left( t,X^{u,\eta }\left( t\right) ,\mathbb{E}\left[
X^{u,\eta }\left( t\right) \right] ,u^{\varepsilon }\left( t\right) \right)
-b\left( t,X^{u,\eta }\left( t\right) ,\mathbb{E}\left[ X^{u,\eta }\left(
t\right) \right] ,u\left( t\right) \right) \right] \text{d}t \\ 
\qquad +\left[ \sigma _{x}\left( t,X^{u,\eta }\left( t\right) ,\mathbb{E}%
\left[ X^{u,\eta }\left( t\right) \right] \right) y^{1}\left( t\right)
\right. \\ 
\qquad \left. +\sigma _{\tilde{x}}\left( t,X^{u,\eta }\left( t\right) ,%
\mathbb{E}\left[ X^{u,\eta }\left( t\right) \right] \right) \mathbb{E}\left[
y^{1}\left( t\right) \right] \right] \text{d}W\left( t\right) , \\ 
y^{1}\left( 0\right) =0.%
\end{array}%
\right.  \tag{3.1.4}
\end{equation}

Owing to (H1)-(H3), it is fairly east to check that (3.1.4) has a unique
solution. The following lemma plays important roles to establish the
inequality.

\begin{lemma}
Assume that (H1)-(H3). Then we have%
\begin{equation*}
\mathbb{E}\left[ \int_{0}^{T}\left\vert y^{1}\left( t\right) \right\vert ^{2}%
\text{d}t\right] \leq o\left( \varepsilon \right) .
\end{equation*}
\end{lemma}

The proof is classical. We omit it. Now let us define the Hamiltonian
associated with random variables $X\in L^{1}\left( \Omega ,\mathcal{F}%
,P\right) $ as follows:%
\begin{equation*}
H\left( t,X,v,p,q\right) :=b\left( t,X,\mathbb{E}\left[ X\right] ,u\right)
p+\sigma \left( t,X,\mathbb{E}\left[ X\right] \right) q+f\left( t,X,\mathbb{E%
}\left[ X\right] ,u\right)
\end{equation*}%
for $\left( p,q\right) \in \mathbb{R}^{n}\mathbb{\times R}^{n\times m},$ and
introduce the adjoint equations involved in the stochastic maximum principle
for our control problem. Note that $\sigma $ does not contain control
variable. So the first order adjoint equation is the following linear
backward SDEs of mean-field type%
\begin{equation}
\left\{ 
\begin{array}{l}
\text{d}p\left( t\right) =-\left[ b_{x}\left( t,X^{u,\eta }\left( t\right) ,%
\mathbb{E}\left[ X^{u,\eta }\left( t\right) \right] ,u\left( t\right)
\right) p\left( t\right) \right. \\ 
\qquad +\mathbb{E}\left[ b_{y}\left( t,X^{u,\eta }\left( t\right) ,\mathbb{E}%
\left[ X^{u,\eta }\left( t\right) \right] ,u\left( t\right) \right) p\left(
t\right) \right] \\ 
\qquad +\sigma _{x}\left( t,X^{u,\eta }\left( t\right) ,\mathbb{E}\left[
X^{u,\eta }\left( t\right) \right] \right) q\left( t\right) \\ 
\qquad +\mathbb{E}\left[ \sigma _{y}\left( t,X^{u,\eta }\left( t\right) ,%
\mathbb{E}\left[ X^{u,\eta }\left( t\right) \right] \right) q\left( t\right) %
\right] \\ 
\qquad -f_{x}\left( t,X^{u,\eta }\left( t\right) ,\mathbb{E}\left[ X^{u,\eta
}\left( t\right) \right] ,u\left( t\right) \right) \\ 
\qquad \left. -\mathbb{E}\left[ f_{y}\left( t,X^{u,\eta }\left( t\right) ,%
\mathbb{E}\left[ X^{u,\eta }\left( t\right) \right] ,u\left( t\right)
\right) \right] \right] \text{d}t+q\left( t\right) \text{d}W\left( t\right) ,
\\ 
p\left( T\right) =h_{x}\left( X^{u,\eta }\left( T\right) ,\mathbb{E}\left[
X^{u,\eta }\left( T\right) \right] \right) \\ 
\qquad +\mathbb{E}\left[ h_{y}\left( X^{u,\eta }\left( T\right) ,\mathbb{E}%
\left[ X^{u,\eta }\left( T\right) \right] \right) \right] .%
\end{array}%
\right.  \tag{3.1.5}
\end{equation}%
Thanks to Theorem 3.1. in Buckdahn, Li and Peng [11], under the assumption
(H1), (3.1.5) admits a unique $\mathcal{F}$-adapted solution $\left( p\left(
\cdot \right) ,q\left( \cdot \right) \right) $ such that 
\begin{equation*}
\mathbb{E}\left[ \sup\limits_{t\in \left[ 0,T\right] }\left\vert p\left(
t\right) \right\vert ^{2}\right] +\mathbb{E}\left[ \int_{0}^{T}\left\vert
q\left( t\right) \right\vert ^{2}\text{d}t\right] <+\infty .
\end{equation*}

\begin{theorem}
Let (H1)-(H3) hold. If $\left( X^{u,\eta }\left( \cdot \right) ,u\left(
\cdot \right) ,\eta \left( \cdot \right) \right) $ is an optimal solution of
(2.0.1), then there exist a pair of $\mathcal{F}$-adapted processes $\left(
p\left( \cdot \right) ,q\left( \cdot \right) \right) $ satisfying (3.1.5)
such that%
\begin{equation}
H\left( t,X^{u,\eta }\left( t\right) ,v,\eta \left( t\right) ,p\left(
t\right) ,q\left( t\right) \right) -H\left( t,X^{u,\eta }\left( t\right)
,u\left( t\right) ,\eta \left( t\right) ,p\left( t\right) ,q\left( t\right)
\right) \geq 0,  \tag{3.1.6}
\end{equation}%
\begin{equation}
P\left\{ \varphi _{i}\left( t\right) +G_{i}\left( t\right) p\left( t\right)
\geq 0\right\} =1,  \tag{3.1.7}
\end{equation}%
\begin{equation}
P\left\{ \sum_{i=1}^{m}\mathbf{I}_{\varphi _{i}\left( t\right) +G_{i}\left(
t\right) p\left( t\right) d\eta _{i}\left( t\right) \geq 0}=0\right\} =1 
\tag{3.1.8}
\end{equation}%
for all $v\in U,$ a.e. $t\in \left[ 0,T\right] ,$ $P$-a.s.
\end{theorem}

The proof of (3.1.6) can be seen in [12], Theorem 2.1. without control
variable in diffusion term. To prove (3.1.7) and (3.1.8), we need the
following lemmas. At the beginning, we introduce the convex perturbation%
\begin{equation*}
\left( u\left( t\right) ,\eta ^{\alpha }\left( t\right) \right) =\left(
u\left( t\right) ,\eta \left( t\right) +\alpha \left( \xi \left( t\right)
+\eta \left( t\right) \right) \right)
\end{equation*}%
where $\alpha \in \left[ 0,1\right] $ and $\xi \left( \cdot \right) $ is an
arbitrary element of $\mathcal{U}_{2}.$ Suppose that $\left( u\left( \cdot
\right) ,\eta \left( \cdot \right) \right) $ is an optimal control, we will
derive the second variational inequality from the fact that 
\begin{equation*}
J\left( \left( u\left( t\right) ,\eta ^{\alpha }\left( t\right) \right)
\right) -J\left( u\left( t\right) ,\eta \left( t\right) \right) \geq 0.
\end{equation*}

\begin{lemma}
Under the assumptions (H1)-(H3), we have%
\begin{equation*}
\lim\limits_{\alpha \rightarrow 0}\mathbb{E}\left[ \sup\limits_{t\in \left[
0,T\right] }\left\vert X^{u,\eta ^{\alpha }}\left( t\right) -X^{u,\eta
}\left( t\right) \right\vert ^{2}\right] =0.
\end{equation*}
\end{lemma}

\begin{proof}
From standard estimates and the Burkholder-Davis-Gundy inequality we have 
\begin{eqnarray*}
&&\mathbb{E}\left[ \left\vert X^{u,\eta ^{\alpha }}\left( t\right)
-X^{u,\eta }\left( t\right) \right\vert ^{2}\right] \\
&\leq &3T\mathbb{E}\left[ \int_{0}^{t}\left\vert b\left( s,X^{u,\eta
^{\alpha }}\left( s\right) ,\mathbb{E}\left[ X^{u,\eta ^{\alpha }}\left(
s\right) \right] ,u\left( s\right) \right) -b\left( s,X^{u,\eta }\left(
s\right) ,\mathbb{E}\left[ X^{u,\eta }\left( s\right) \right] ,u\left(
s\right) \right) \right\vert ^{2}\text{d}s\right] \\
&&+3\mathbb{E}\left[ \int_{0}^{t}\left\vert \sigma \left( s,X^{u,\eta
^{\alpha }}\left( s\right) ,\mathbb{E}\left[ X^{u,\eta ^{\alpha }}\left(
s\right) \right] \right) -\sigma \left( s,X^{u,\eta }\left( s\right) ,%
\mathbb{E}\left[ X^{u,\eta }\left( s\right) \right] \right) \right\vert ^{2}%
\text{d}s\right] \\
&&+3\alpha ^{2}\mathbb{E}\left[ \int_{0}^{t}\left\vert G\left( s\right)
\left( \xi \left( s\right) -\eta \left( s\right) \right) \right\vert ^{2}%
\text{d}s\right] \\
&\leq &C\mathbb{E}\left[ \int_{0}^{t}\left\vert X^{u,\eta ^{\alpha }}\left(
s\right) -X^{u,\eta }\left( s\right) \right\vert ^{2}\text{d}s\right]
+C\alpha ^{2}\mathbb{E}\left\vert \xi \left( T\right) -\eta \left( T\right)
\right\vert ^{2},\text{ }t\in \left[ 0,T\right] .
\end{eqnarray*}%
where $C$ depends on $T,$ and the Lipschitz coefficients of $b,\sigma .$
From Gronwall's lemma we have the desired result.
\end{proof}

We now introduce the following variational equations of (2.0.1):%
\begin{equation}
\left\{ 
\begin{array}{l}
dy^{2}\left( t\right) =b_{x}\left( t,X^{u,\eta }\left( t\right) ,\mathbb{E}%
\left[ X^{u,\eta }\left( t\right) \right] ,u\left( t\right) \right)
y^{2}\left( t\right) \text{d}t \\ 
\qquad \qquad +b_{y}\left( t,X^{u,\eta }\left( t\right) ,\mathbb{E}\left[
X^{u,\eta }\left( t\right) \right] ,u\left( t\right) \right) q\left( t,\text{%
d}a\right) \mathbb{E}\left[ y^{2}\left( t\right) \right] \text{d}t \\ 
\qquad \qquad +\sigma _{x}\left( t,X^{u,\eta }\left( t\right) ,\mathbb{E}%
\left[ X^{u,\eta }\left( t\right) \right] \right) y^{2}\left( t\right) \text{%
d}W\left( t\right) \\ 
\qquad \qquad +\sigma _{y}\left( t,X^{u,\eta }\left( t\right) ,\mathbb{E}%
\left[ X^{u,\eta }\left( t\right) \right] \right) \mathbb{E}\left[
y^{2}\left( t\right) \right] \text{d}W\left( t\right) \\ 
\qquad \qquad +G\left( t\right) \left( \xi \left( t\right) -\eta \left(
t\right) \right) \text{d}t, \\ 
y^{2}\left( 0\right) =0.%
\end{array}%
\right.  \tag{3.1.9}
\end{equation}%
From (H1)-(H2) it is easy to check that (3.1.9) has a unique strong
solution. Moreover, we have

\begin{lemma}
Under the assumptions (H1)-(H3), we have%
\begin{equation*}
\lim\limits_{\alpha \rightarrow 0}\mathbb{E}\left[ \left\vert \frac{%
X^{u,\eta ^{\alpha }}\left( t\right) -X^{u,\eta }\left( t\right) }{\alpha }%
-y^{2}\left( t\right) \right\vert ^{2}\right] =0,\text{ }t\in \left[ 0,T%
\right] .
\end{equation*}
\end{lemma}

\begin{proof}
We have%
\begin{eqnarray*}
&&\frac{X^{u,\eta ^{\alpha }}\left( t\right) -X^{u,\eta }\left( t\right) }{%
\alpha }-y^{2}\left( t\right) \\
&=&\frac{1}{\alpha }\int_{0}^{t}b\left( s,X^{u,\eta ^{\alpha }}\left(
s\right) ,\mathbb{E}\left[ X^{u,\eta ^{\alpha }}\left( s\right) \right]
,u\left( s\right) \right) \text{d}s \\
&&-\frac{1}{\alpha }\int_{0}^{t}b\left( s,X^{u,\eta }\left( s\right) ,%
\mathbb{E}\left[ X^{u,\eta ^{\alpha }}\left( s\right) \right] ,u\left(
s\right) \right) \text{d}s \\
&&+\frac{1}{\alpha }\int_{0}^{t}b\left( s,X^{u,\eta }\left( s\right) ,%
\mathbb{E}\left[ X^{u,\eta ^{\alpha }}\left( s\right) \right] ,u\left(
s\right) \right) \text{d}s \\
&&-\frac{1}{\alpha }\int_{0}^{t}b\left( s,X^{u,\eta }\left( s\right) ,%
\mathbb{E}\left[ X^{u,\eta }\left( s\right) \right] ,u\left( s\right)
\right) \text{d}s \\
&&+\frac{1}{\alpha }\int_{0}^{t}\sigma \left( s,X^{u,\eta ^{\alpha }}\left(
s\right) ,\mathbb{E}\left[ X^{u,\eta ^{\alpha }}\left( s\right) \right]
,u\left( s\right) \right) \text{d}W\left( s\right) \\
&&-\frac{1}{\alpha }\int_{0}^{t}\sigma \left( s,X^{u,\eta }\left( s\right) ,%
\mathbb{E}\left[ X^{u,\eta ^{\alpha }}\left( s\right) \right] ,u\left(
s\right) \right) \text{d}W\left( s\right) \\
&&+\frac{1}{\alpha }\int_{0}^{t}\sigma \left( s,X^{u,\eta }\left( s\right) ,%
\mathbb{E}\left[ X^{u,\eta ^{\alpha }}\left( s\right) \right] ,u\left(
s\right) \right) \text{d}W\left( s\right) \\
&&-\frac{1}{\alpha }\int_{0}^{t}\sigma \left( s,X^{u,\eta ^{\alpha }}\left(
s\right) ,\mathbb{E}\left[ X^{u,\eta ^{\alpha }}\left( s\right) \right]
,u\left( s\right) \right) \text{d}W\left( s\right) \\
&&-\int_{0}^{t}b_{x}\left( t,X^{u,\eta }\left( t\right) ,\mathbb{E}\left[
X^{u,\eta }\left( t\right) \right] ,u\left( s\right) \right) y^{2}\left(
t\right) \text{d}t \\
&&-\int_{0}^{t}b_{y}\left( t,X^{u,\eta }\left( t\right) ,\mathbb{E}\left[
X^{u,\eta }\left( t\right) \right] ,u\left( s\right) \right) \mathbb{E}\left[
y^{2}\left( t\right) \right] \text{d}t \\
&&-\int_{0}^{t}\sigma _{x}\left( t,X^{u,\eta }\left( t\right) ,\mathbb{E}%
\left[ X^{u,\eta }\left( t\right) \right] \right) y^{2}\left( t\right) \text{%
d}W\left( t\right) \\
&&-\int_{0}^{t}\sigma _{y}\left( t,X^{u,\eta }\left( t\right) ,\mathbb{E}%
\left[ X^{u,\eta }\left( t\right) \right] \right) \mathbb{E}\left[
y^{2}\left( t\right) \right] \text{d}W\left( t\right) .
\end{eqnarray*}%
Set $\Sigma \left( t\right) =\frac{X^{u,\eta ^{\alpha }}\left( t\right)
-X^{u,\eta }\left( t\right) }{\alpha }-y^{2}\left( t\right) ,$ $t\in \left[
0,T\right] .$ Taking the expectation, we have%
\begin{eqnarray*}
&&\mathbb{E}\left[ \left\vert \Sigma \left( t\right) \right\vert ^{2}\right]
\\
&=&C\mathbb{E}\left[ \int_{0}^{t}\int_{0}^{1}\left\vert \bar{b}_{x}\left(
s\right) \Sigma \left( s\right) \right\vert ^{2}\text{d}s\text{d}\theta %
\right] \\
&&+C\mathbb{E}\left[ \int_{0}^{t}\int_{0}^{1}\left\vert \bar{b}_{y}\left(
s\right) \mathbb{E}\left[ \Sigma \left( s\right) \right] \right\vert ^{2}%
\text{d}s\text{d}\theta \right] \\
&&+C\mathbb{E}\left[ \int_{0}^{t}\int_{0}^{1}\left\vert \bar{\sigma}%
_{x}\left( s\right) \Sigma \left( s\right) \right\vert ^{2}\text{d}W\left(
s\right) \text{d}\theta \right] \\
&&+C\mathbb{E}\left[ \int_{0}^{t}\int_{0}^{1}\left\vert \bar{\sigma}%
_{y}\left( s\right) \mathbb{E}\left[ \Sigma \left( s\right) \right]
\right\vert ^{2}\text{d}W\left( s\right) \text{d}\theta \right] \\
&&+C\mathbb{E}\left[ \left\vert \kappa ^{\varrho }\left( t\right)
\right\vert ^{2}\right] ,
\end{eqnarray*}%
where%
\begin{equation*}
\left\{ 
\begin{array}{l}
\bar{b}_{x}\left( s\right) =b_{x}\left( s,X^{u,\eta }\left( s\right) +\theta
\varrho \left( \Sigma \left( s\right) -y^{2}\left( s\right) \right) ,\mathbb{%
E}\left[ X^{u,\eta ^{\varrho }}\left( s\right) \right] ,a\right) , \\ 
\bar{b}_{y}\left( s\right) =b_{y}\left( s,X^{u,\eta }\left( s\right) ,%
\mathbb{E}\left[ X^{u,\eta }\left( s\right) \right] +\theta \varrho \left( 
\mathbb{E}\left[ \Sigma \left( s\right) -y^{2}\left( s\right) \right]
\right) ,a\right) , \\ 
\bar{\sigma}_{x}\left( s\right) =\sigma _{x}\left( s,X^{u,\eta }\left(
s\right) +\theta \varrho \left( \Sigma \left( s\right) -y^{2}\left( s\right)
\right) ,\mathbb{E}\left[ X^{u,\eta ^{\varrho }}\left( s\right) \right]
\right) , \\ 
\bar{\sigma}_{y}\left( s\right) =\sigma _{y}\left( s,X^{u,\eta }\left(
s\right) ,\mathbb{E}\left[ X^{u,\eta }\left( s\right) \right] +\theta
\varrho \left( \mathbb{E}\left[ \Sigma \left( s\right) -y^{2}\left( s\right) %
\right] \right) \right) ,%
\end{array}%
\right.
\end{equation*}%
and%
\begin{eqnarray*}
&&\left\vert \kappa ^{\varrho }\left( t\right) \right\vert ^{2} \\
&=&\int_{0}^{t}\int_{0}^{1}\bar{b}_{x}\left( s\right) y^{2}\left( s\right) 
\text{d}s\text{d}\theta \\
&&+\int_{0}^{t}\int_{0}^{1}\bar{b}_{y}\left( s\right) \mathbb{E}\left[
y^{2}\left( s\right) \right] \text{d}s\text{d}\theta \\
&&+\int_{0}^{t}\int_{0}^{1}\bar{\sigma}_{x}\left( s\right) y^{2}\left(
s\right) \text{d}W\left( s\right) \text{d}\theta \\
&&+\int_{0}^{t}\int_{0}^{1}\bar{\sigma}_{y}\left( s\right) \mathbb{E}\left[
y^{1}\left( s\right) \right] \text{d}W\left( s\right) \text{d}\theta \\
&&-\int_{0}^{t}b_{x}\left( s,X^{q,\eta }\left( s\right) ,\mathbb{E}\left[
X^{q,\eta }\left( s\right) \right] ,u\left( s\right) \right) y^{2}\left(
s\right) \text{ds} \\
&&-\int_{0}^{t}b_{y}\left( s,X^{q,\eta }\left( s\right) ,\mathbb{E}\left[
X^{q,\eta }\left( s\right) \right] ,a\right) \mathbb{E}\left[ y^{2}\left(
s\right) \right] \text{d}s \\
&&-\int_{0}^{t}\sigma _{x}\left( s,X^{q,\eta }\left( s\right) ,\mathbb{E}%
\left[ X^{q,\eta }\left( s\right) \right] \right) y^{2}\left( s\right) \text{%
d}W\left( s\right) \\
&&-\int_{0}^{t}\sigma _{y}\left( s,X^{q,\eta }\left( s\right) ,\mathbb{E}%
\left[ X^{q,\eta }\left( s\right) \right] \right) \mathbb{E}\left[
y^{2}\left( s\right) \right] \text{d}W\left( s\right) .
\end{eqnarray*}%
By (H1), we get%
\begin{equation*}
\mathbb{E}\left[ \left\vert \Sigma \left( t\right) \right\vert ^{2}\right]
\leq C\mathbb{E}\left[ \int_{0}^{t}\left\vert \Sigma \left( s\right)
\right\vert ^{2}\text{d}s\right] +C\mathbb{E}\left[ \left\vert \kappa
^{\varrho }\left( t\right) \right\vert ^{2}\right] .
\end{equation*}%
Noting that 
\begin{equation*}
\lim\limits_{\varrho \rightarrow 0}\mathbb{E}\left[ \left\vert \kappa
^{\varrho }\left( t\right) \right\vert ^{2}\right] =0.
\end{equation*}%
By Gronwall's lemma, we get the desired result.
\end{proof}

Now we give the variational inequality.

\begin{lemma}
Assume that (H1)-(H3) hold. Then we have%
\begin{eqnarray*}
0 &\leq &\mathbb{E}\left[ h_{x}\left( X^{u,\eta }\left( T\right) ,\mathbb{E}%
\left[ X^{u,\eta }\left( T\right) \right] \right) y^{2}\left( T\right)
+h_{y}\left( X^{u,\eta }\left( T\right) ,\mathbb{E}\left[ X^{u,\eta }\left(
T\right) \right] \right) \mathbb{E}\left[ y^{2}\left( T\right) \right] %
\right] \\
&&+\mathbb{E}\left[ \int_{0}^{T}f_{x}\left( t,X^{u,\eta }\left( t\right) ,%
\mathbb{E}\left[ X^{u,\eta }\left( t\right) \right] ,u\left( t\right)
\right) y^{2}\left( t\right) \text{d}t\right] \\
&&+\mathbb{E}\left[ \int_{0}^{T}f_{y}\left( t,X^{u,\eta }\left( t\right) ,%
\mathbb{E}\left[ X^{u,\eta }\left( t\right) \right] ,u\left( t\right)
\right) \mathbb{E}\left[ y^{2}\left( t\right) \right] \text{d}t\right] \\
&&+\mathbb{E}\left[ \int_{0}^{T}\varphi \left( t\right) \text{d}\left( \xi
\left( t\right) -\eta \left( t\right) \right) \right] .
\end{eqnarray*}%
\begin{equation}
\tag{3.1.10}
\end{equation}
\end{lemma}

\begin{proof}
From Lemma 3, we have%
\begin{eqnarray*}
&&\lim\limits_{\alpha \rightarrow 0}\mathbb{E}\left[ \int_{0}^{T}f\left(
t,X^{u,\eta ^{\alpha }}\left( t\right) ,\mathbb{E}\left[ X^{u,\eta ^{\alpha
}}\left( t\right) \right] ,u\left( t\right) \right) y^{2}\left( t\right) 
\text{d}t\right] \\
&&-\mathbb{E}\left[ \int_{0}^{T}f\left( t,X^{u,\eta }\left( t\right) ,%
\mathbb{E}\left[ X^{u,\eta }\left( t\right) \right] ,u\left( t\right)
\right) y^{2}\left( t\right) \text{d}t\right] \\
&=&\lim\limits_{\alpha \rightarrow 0}\mathbb{E}\left[ \int_{0}^{T}f\left(
t,X^{u,\eta ^{\alpha }}\left( t\right) ,\mathbb{E}\left[ X^{u,\eta ^{\alpha
}}\left( t\right) \right] ,u\left( t\right) \right) y^{2}\left( t\right) 
\text{d}t\right] \\
&&-\mathbb{E}\left[ \int_{0}^{T}f\left( t,X^{u,\eta }\left( t\right) ,%
\mathbb{E}\left[ X^{u,\eta ^{\alpha }}\left( t\right) \right] ,u\left(
t\right) \right) y^{2}\left( t\right) \text{d}t\right] \\
&&+\mathbb{E}\left[ \int_{0}^{T}f\left( t,X^{u,\eta }\left( t\right) ,%
\mathbb{E}\left[ X^{u,\eta ^{\alpha }}\left( t\right) \right] ,u\left(
t\right) \right) y^{2}\left( t\right) \text{d}t\right] \\
&&-\mathbb{E}\left[ \int_{0}^{T}f\left( t,X^{q,\eta }\left( t\right) ,%
\mathbb{E}\left[ X^{q,\eta }\left( t\right) \right] ,u\left( t\right)
\right) y^{2}\left( t\right) \text{d}t\right] \\
&=&\mathbb{E}\left[ \int_{0}^{T}f_{x}\left( t,X^{q,\eta }\left( t\right) ,%
\mathbb{E}\left[ X^{q,\eta }\left( t\right) \right] ,a\right) y^{2}\left(
t\right) \text{d}t\right] \\
&&+\mathbb{E}\left[ \int_{0}^{T}f_{y}\left( t,X^{q,\eta }\left( t\right) ,%
\mathbb{E}\left[ X^{q,\eta }\left( t\right) \right] ,a\right) \mathbb{E}%
\left[ y^{2}\left( t\right) \right] \text{d}t\right] .
\end{eqnarray*}%
The same method to deal with $h,$ from the fact that 
\begin{equation*}
\frac{J\left( \left( q^{\varrho }\left( \cdot \right) ,\eta ^{\varrho
}\left( \cdot \right) \right) \right) -J\left( q\left( \cdot \right) ,\eta
\left( \cdot \right) \right) }{\varrho }\geq 0.
\end{equation*}%
We get the desired result.
\end{proof}

\begin{lemma}
Let $\left( u,\eta \right) $ be a pair of optimal control and let $X^{u,\eta
}\left( \cdot \right) $ be the corresponding trajectory. Then we have 
\begin{equation}
0\leq \mathbb{E}\left[ \int_{0}^{T}\left( \varphi \left( t\right) +p\left(
t\right) G\left( t\right) \right) \text{d}\left( \xi \left( t\right) -\eta
\left( t\right) \right) \right] .  \tag{3.1.11}
\end{equation}
\end{lemma}

\begin{proof}
Applying It\^{o}'s formula to $\left\langle y^{1}\left( t\right) ,p\left(
t\right) \right\rangle $ on $\left[ 0,T\right] ,$ we have%
\begin{eqnarray*}
&&\mathbb{E}\left[ h_{x}\left( X^{u,\eta }\left( T\right) ,\mathbb{E}\left[
X^{u,\eta }\left( T\right) \right] \right) y^{2}\left( T\right) +\mathbb{E}%
\left[ h_{y}\left( X^{u,\eta }\left( T\right) ,\mathbb{E}\left[ X^{u,\eta
}\left( T\right) \right] \right) \right] \left[ y^{2}\left( T\right) \right] %
\right]  \\
&&+\mathbb{E}\left[ \int_{0}^{T}f_{x}\left( t,X^{u,\eta }\left( t\right) ,%
\mathbb{E}\left[ X^{u,\eta }\left( t\right) \right] ,a\right) y^{2}\left(
t\right) \text{d}t\right]  \\
&&+\mathbb{E}\left[ \int_{0}^{T}f_{y}\left( t,X^{u,\eta }\left( t\right) ,%
\mathbb{E}\left[ X^{u,\eta }\left( t\right) \right] ,a\right) \mathbb{E}%
\left[ y^{2}\left( t\right) \right] \text{d}t\right]  \\
&&+\mathbb{E}\left[ \int_{0}^{T}\varphi \left( t\right) \text{d}\left( \xi
\left( t\right) -\eta \left( t\right) \right) \right]  \\
&=&\mathbb{E}\left[ \int_{0}^{T}\left( \varphi \left( t\right) +p\left(
t\right) G\left( t\right) \right) \text{d}\left( \xi \left( t\right) -\eta
\left( t\right) \right) \right] .
\end{eqnarray*}%
From Lemma 5, we get the desired result.
\end{proof}

Now we are able to give the proof of Theorem 2.

\begin{proof}
Proof of Theorem 1

: (3.1.6) can be seen in [12], Theorem 2.1. With the help of (3.1.11), the
proof of (3.1.7), (3.1.8) is going exactly as Theorem 3.7 in [8]. The proof
is complete.
\end{proof}

\section{Relaxed Singular Optimal Control Problem}

\subsection{Relaxed controls model}

In this subsection, we set up the relaxed model. Before that, we give an
example to illustrate our motivation.

\begin{example}
Let $U=\left\{ -1,1\right\} ,$%
\begin{equation*}
\mathcal{U=}\left\{ \left. v\left( \cdot \right) :\left[ 0,1\right]
\rightarrow U\right\vert v\left( \cdot \right) \text{ measurable}\right\} ,
\end{equation*}%
and%
\begin{equation*}
J\left( v\right) =\int_{0}^{1}\left( y^{0,v}\left( t\right) \right) ^{2}%
\text{d}t,
\end{equation*}%
where $y^{0,v}\left( t\right) $ denotes the solution of 
\begin{equation*}
\left\{ 
\begin{array}{l}
\text{d}y^{0,v}\left( t\right) =v\left( t\right) \text{d}t \\ 
y^{0,v}\left( 0\right) =0,\text{ }t\in \left[ 0,1\right] .%
\end{array}%
\right.
\end{equation*}%
The optimal control problem is that

\textbf{Problem}: Find a pair $\left( \bar{y}\left( \cdot \right) ,\bar{v}%
\left( \cdot \right) \right) $ such that 
\begin{equation*}
J\left( \bar{v}\right) =\inf\limits_{v\left( \cdot \right) \in \mathcal{U}%
}J\left( v\right) .
\end{equation*}%
We will show that 
\begin{equation*}
J\left( \bar{v}\right) =\inf\limits_{v\left( \cdot \right) \in \mathcal{U}%
}J\left( v\right) =0.
\end{equation*}%
Indeed, for any $n>0,$ let%
\begin{equation*}
v^{n}\left( t\right) =\left( -1\right) ^{k},\qquad \frac{k}{n}\leq t<\frac{%
k+1}{n},0\leq k\leq n-1.
\end{equation*}%
Then immediately, we have%
\begin{equation*}
\left\vert y^{0,v^{n}}\left( t\right) \right\vert ^{2}\leq \frac{1}{n^{2}}%
,\qquad J\left( v^{n}\right) \leq \frac{1}{n^{2}}.
\end{equation*}%
On the other hand, for any $v\in \mathcal{U}$, 
\begin{equation*}
J\left( v\right) =\int_{0}^{1}\left( y^{0,v}\left( t\right) \right) ^{2}%
\text{d}t\geq 0.
\end{equation*}%
Consequently, we derive that%
\begin{equation*}
J\left( \bar{v}\right) =0.
\end{equation*}%
However, the infimum $0$ could not be achieved. To see this, let $\left( 
\bar{y}\left( \cdot \right) ,\bar{v}\left( \cdot \right) \right) $ be the
optimal pair. Then 
\begin{equation*}
\bar{y}^{0,\bar{v}}\left( t\right) =\bar{v}\left( t\right) =0,
\end{equation*}
which is impossible. As a matter of fact, Let $\delta _{u}$ denote the
atomic measure concentrated at a single point $u.$ Then 
\begin{equation*}
\text{d}t\delta _{v^{n}\left( t\right) }\text{d}a\rightarrow \frac{1}{2}%
\text{d}t\left( \delta _{-1}+\delta _{1}\right) \text{d}a,\text{ }t\in \left[
0,1\right] .
\end{equation*}
\end{example}

The above example shows that the strict control problem defined in section
3, may fail to have an optimal solution. The reason is that the compact set $%
U$ of strict controls is too narrow and should be embedded into a wider
class with a richer topological structure for which the control problem
becomes solvable. Our main goal in this section is to establish a maximum
principle for relaxed-singular controls. This leads to necessary conditions
satisfied by an optimal relaxed-singular control, which exists under general
assumptions on the coefficients.

The idea of relaxed singular controls is to replace the $U$-valued process $%
u\left( t\right) $ with $P(U)$-valued process $q\left( t\right) $, where $%
P(U)$ is the space of probability measures equipped with the topology of
weak convergence (more information see in [7]).

\begin{definition}
A relaxed control is the term 
\begin{equation*}
q=\left( \Omega ,\mathcal{F},\left( \mathcal{F}_{t}\right) _{t\geq
0},P,W\left( t\right) ,q\left( t\right) ,\chi \left( t\right) ,\xi \right) 
\end{equation*}
such that

(1) $\left( \Omega ,\mathcal{F},\left( \mathcal{F}_{t}\right) _{t\geq
0},P\right) $ is a filtered probability space the usual conditions.

(2) $q\left( t\right) $ is a $P\left( \bar{U}\right) $-valued process,
progressively measurable with respect to $\left( \mathcal{F}_{t}\right)
_{t\geq 0}$ and such that for each $t,$ $\mathbf{I}_{\left( 0,t\right]
}\cdot q$ is $\mathcal{F}_{t}$-measurable.

(3) $\chi \left( t\right) $ is $\mathbb{R}^{n}$-valued and $\mathcal{F}_{t}$%
-adapted with continuous paths such that $\chi \left( 0\right) =\xi $ and
for each $f\in C_{b}^{2}\left( \mathbb{R}^{n};\mathbb{R}\right) $%
\begin{equation}
f\left( \chi \left( t\right) \right) -f\left( \xi \right)
-\int_{0}^{t}\int_{U}\mathcal{L}f\left( s,\chi \left( s\right) ,a\right)
q_{s}\left( \omega ,\text{d}a\right) \text{d}s  \tag{4.1.1}
\end{equation}%
is a $P$-martingale, where $\mathcal{L}$ is the infinitesimal generator.
\end{definition}

Obviously, The set of strict controls is embedded into the set of relaxed
controls by the mapping%
\begin{equation*}
u\rightarrow \text{d}t\delta _{u\left( t\right) }\text{d}a,\text{ }t\geq 0.
\end{equation*}

\begin{definition}
An admissible relaxed control $q$ is a relaxed control such that 
\begin{equation*}
\mathbb{E}\left[ \sup\limits_{t\in \left[ 0,T\right] }\left\vert q\left(
t\right) \right\vert ^{2}\right] <+\infty .
\end{equation*}
\end{definition}

We denote by $\mathcal{R}_{1}$ the set of all admissible relaxed controls
controls and denote by $\mathcal{R=R}_{1}\times \mathcal{U}_{2}$ the set of
relaxed-singular controls. We now introduce the following relaxed-singular
SDEs 
\begin{equation}
\left\{ 
\begin{array}{l}
dX^{q,\eta }\left( t\right) =\int_{U}b\left( t,X^{q,\eta }\left( t\right) ,%
\mathbb{E}\left[ X^{q,\eta }\left( t\right) \right] ,a\right) q\left( t,%
\text{d}a\right) \text{d}t \\ 
\qquad +\sigma \left( t,X^{q,\eta }\left( t\right) ,\mathbb{E}\left[
X^{q,\eta }\left( t\right) \right] \right) \text{d}W\left( t\right) +G\left(
t\right) d\eta \left( t\right) , \\ 
X^{q,\eta }\left( 0\right) =x_{0}\in \mathbb{R}^{n},\quad t\in \left[
0,+\infty \right) ,%
\end{array}%
\right.  \tag{4.1.2}
\end{equation}%
and the optimal relaxed singular control cost function 
\begin{eqnarray*}
\mathcal{J}\left( \left( q,\eta \right) \right) &=&\mathbb{E}\left[
\int_{U}\int_{0}^{T}f\left( t,X^{u,\eta }\left( t\right) ,\mathbb{E}\left[
X^{u,\eta }\left( t\right) \right] ,a\right) q\left( t,\text{d}a\right) 
\text{d}t\right. \\
&&\left. +h\left( X^{u,\eta }\left( T\right) ,\mathbb{E}\left[ X^{u,\eta
}\left( T\right) \right] \right) +\int_{0}^{T}k\left( t\right) \eta \left(
t\right) \right] .
\end{eqnarray*}%
A relaxed-singular control $\left( \bar{q},\bar{\eta}\right) $ is called
optimal if it solves%
\begin{equation}
\mathcal{J}\left( \left( \bar{q},\bar{\eta}\right) \right)
=\inf\limits_{\left( q,\eta \right) \in \mathcal{R}_{1}\times \mathcal{U}%
_{2}}\mathcal{J}\left( \left( q,\eta \right) \right) .  \tag{4.1.3}
\end{equation}%
As you have observed that the coefficients of equation (4.1.2) and the
running cost are linear with respect to the relaxed control variable. On the
other hand, we have replaced $\mathcal{U}_{1}$ by a larger space $P\left( 
\mathcal{U}_{1}\right) $ which is convex. Furthermore, it is fairly easy to
check that $l=\int_{U}b\left( t,X^{q,\eta }\left( t\right) ,\mathbb{E}\left[
X^{q,\eta }\left( t\right) \right] ,a\right) q\left( t,\text{d}a\right) $d$%
t, $ $l=b,f$, respectively, satisfy the assumption (H1). Therefore, for any $%
q\in \mathcal{R}_{1},$ SDEs (4.1.2) admit a unique strong solution and the
new cost function is well-defined.

\begin{remark}
Set $q\left( t\right) =\delta _{u\left( t\right) }$ at a single point $%
u\left( t\right) \in U.$ Then for any $t\in \left[ 0,T\right] ,$ we have,
for $l=b,f,$%
\begin{eqnarray*}
&&\int_{U}b\left( t,X^{q,\eta }\left( t\right) ,\mathbb{E}\left[ X^{q,\eta
}\left( t\right) \right] ,a\right) q\left( t,\text{d}a\right) dt \\
&=&\int_{U}b\left( t,X^{q,\eta }\left( t\right) ,\mathbb{E}\left[ X^{q,\eta
}\left( t\right) \right] ,a\right) \delta _{u\left( t\right) }\left( \text{d}%
a\right) dt \\
&=&b\left( t,X^{q,\eta }\left( t\right) ,\mathbb{E}\left[ X^{q,\eta }\left(
t\right) \right] ,u\left( t\right) \right) .
\end{eqnarray*}%
Simultaneously, $X^{q,\eta }\left( t\right) =X^{u,\eta }\left( t\right) $
and $\mathcal{J}\left( \left( q,\eta \right) \right) =J\left( \left( u,\eta
\right) \right) .$ Hence the problem of strict-singular controls problem is
a particular case of relaxed-singular control problem.
\end{remark}

Additionally, throughout this section we suppose that

\begin{enumerate}
\item[\textbf{(H4)}] $b,h$ are bounded.
\end{enumerate}

\begin{lemma}[\textbf{Chattering lemma}]
Let $q\left( \cdot \right) $ be a predictable process with values in the
space of probability measures on $U.$ Then there exists a sequence of
predictable processes $\left( u^{n}\left( \cdot \right) \right) _{n\geq 1}$
with values in $U$ such that the sequence of random measures $\delta
_{u^{n}\left( \cdot \right) }$d$a$d$t$ converges weakly to $q\left( t\right) 
$d$a$d$t,$ $P$-a.s.
\end{lemma}

We now show the stability property of controlled mean-field SDEs with
respect to control variable.

\begin{lemma}
Assume (H1), (H3) and (H4) hold. For any relaxed control $\left( q,\eta
\right) $, let $X^{q,\eta }\left( \cdot \right) $ denote the corresponding
trajectory. Then there exists a sequence $\left( u^{n},\eta \right) _{n\geq
1}\subset \mathcal{U}_{1}\times \mathcal{U}_{2}$ such that%
\begin{equation}
\lim\limits_{n\rightarrow +\infty }\mathbb{E}\left[ \sup\limits_{t\in \left[
0,T\right] }\left\vert X^{u^{n},\eta }\left( t\right) -X^{q,\eta }\left(
t\right) \right\vert ^{2}\right] =0,  \tag{4.1.4}
\end{equation}%
\begin{equation}
\lim\limits_{n\rightarrow +\infty }J\left( u^{n},\eta \right) =J\left(
q,\eta \right) .  \tag{4.1.5}
\end{equation}
\end{lemma}

\begin{proof}
From standard estimates and Burkholder-Davis-Gundy inequality we get that,
for some $C>0$, only depending on $T,$ and the Lipschitz coefficient of $b$, 
$\sigma $:%
\begin{eqnarray*}
&&\mathbb{E}\left[ \sup\limits_{t\in \left[ 0,T\right] }\left\vert
X^{u^{n},\eta }\left( t\right) -X^{q,\eta }\left( t\right) \right\vert ^{2}%
\right] \\
&\leq &5C\mathbb{E}\left[ \int_{0}^{t}\left\vert b\left( s,X^{u^{n},\eta
}\left( s\right) ,\mathbb{E}\left[ X^{u^{n},\eta }\left( s\right) \right]
,u^{n}\left( s\right) \right) \right. \right. \\
&&\left. -b\left( s,X^{q,\eta }\left( s\right) ,\mathbb{E}\left[
X^{u^{n},\eta }\left( s\right) \right] ,u^{n}\left( s\right) \right)
\right\vert ^{2}\text{d}s \\
&&+\int_{0}^{t}\left\vert b\left( s,X^{q,\eta }\left( s\right) ,\mathbb{E}%
\left[ X^{u^{n},\eta }\left( s\right) \right] ,u^{n}\left( s\right) \right)
\right. \\
&&\left. -b\left( s,X^{q,\eta }\left( s\right) ,\mathbb{E}\left[ X^{q,\eta
}\left( s\right) \right] ,u^{n}\left( s\right) \right) \right\vert ^{2}\text{%
d}s \\
&&+5\left\vert \int_{0}^{t}\int_{U}b\left( s,X^{q,\eta }\left( s\right) ,%
\mathbb{E}\left[ X^{q,\eta }\left( s\right) \right] ,a\right) q^{n}\left(
s\right) \text{d}a\text{d}s\right. \\
&&-\left. \left. \int_{0}^{t}\int_{U}b\left( s,X^{q,\eta }\left( s\right) ,%
\mathbb{E}\left[ X^{q,\eta }\left( s\right) \right] ,a\right) q\left(
s\right) \text{d}a\text{d}s\right\vert ^{2}\right] \\
&&+5C\mathbb{E}\left[ \int_{0}^{t}\left\vert \sigma \left( s,X^{u^{n},\eta
}\left( s\right) ,\mathbb{E}\left[ X^{u^{n},\eta }\left( s\right) \right]
\right) \right. \right. \\
&&\left. -\sigma \left( s,X^{q,\eta }\left( s\right) ,\mathbb{E}\left[
X^{u^{n},\eta }\left( s\right) \right] \right) \right\vert ^{2}\text{d}s \\
&&+8\mathbb{E}\left[ \int_{0}^{t}\left\vert \sigma \left( s,X^{q,\eta
}\left( s\right) ,\mathbb{E}\left[ X^{u^{n},\eta }\left( s\right) \right]
\right) \right. \right. \\
&&\left. -\sigma \left( s,X^{q,\eta }\left( s\right) ,\mathbb{E}\left[
X^{q,\eta }\left( s\right) \right] \right) \right\vert ^{2}\text{d}s \\
&\leq &I_{t}^{n}+C\mathbb{E}\int_{0}^{t}\left\vert X^{u^{n},\eta }\left(
s\right) -X^{q,\eta }\left( s\right) \right\vert ^{2}\text{d}s,
\end{eqnarray*}%
where%
\begin{equation*}
q^{n}\left( s\right) \left( \text{d}a\right) =\delta _{u^{n}\left( s\right) }%
\text{d}a,
\end{equation*}%
and 
\begin{eqnarray*}
I_{t}^{n} &=&\mathbb{E}\left\vert \int_{0}^{t}\int_{U}b\left( s,X^{q,\eta
}\left( s\right) ,\mathbb{E}\left[ X^{q,\eta }\left( s\right) \right]
,a\right) q^{n}\left( s\right) \left( \text{d}a\right) \text{d}s\right. \\
&&\left. -\int_{0}^{t}\int_{U}b\left( s,X^{q,\eta }\left( s\right) ,\mathbb{E%
}\left[ X^{q,\eta }\left( s\right) \right] ,a\right) q\left( s\right) \left( 
\text{d}a\right) \text{d}s\right\vert ^{2}.
\end{eqnarray*}%
Since $b$ is bounded and continuous, and by Lemma 11, using the dominated
convergence theorem, we get 
\begin{equation*}
\lim\limits_{n\rightarrow +\infty }I_{t}^{n}=0.
\end{equation*}%
The main result follows from Gronwall's inequality. Similarly, since $f$, $h$
are Lipschitz continuous in $x,y,$ by Cauchy-Schwarz inequality we have 
\begin{eqnarray*}
&&\left\vert J\left( q^{n}\left( \cdot \right) ,\eta \left( \cdot \right)
\right) -J\left( q\left( \cdot \right) ,\eta \left( \cdot \right) \right)
\right\vert \\
&\leq &C\left( \mathbb{E}\left[ \left\vert X^{u^{n},\eta }\left( T\right)
-X^{q,\eta }\left( T\right) \right\vert ^{2}\right] \right) ^{\frac{1}{2}} \\
&&+C\int_{0}^{T}\left( \mathbb{E}\left[ \left\vert X^{u^{n},\eta }\left(
s\right) -X^{q,\eta }\left( s\right) \right\vert ^{2}\right] \right) ^{\frac{%
1}{2}}\text{d}s \\
&&+\left( \mathbb{E}\left\vert \int_{0}^{T}\int_{U}h\left( s,X^{q,\eta
}\left( s\right) ,\mathbb{E}\left[ X^{q,\eta }\left( s\right) \right]
,a\right) q^{n}\left( s\right) \left( \text{d}a\right) \text{d}s\right.
\right. \\
&&\left. \left. -\int_{0}^{T}\int_{U}h\left( s,X^{q,\eta }\left( s\right) ,%
\mathbb{E}\left[ X^{q,\eta }\left( s\right) \right] ,a\right) q\left(
s\right) \left( \text{d}a\right) \text{d}s\right\vert ^{2}\right) ^{\frac{1}{%
2}}.
\end{eqnarray*}%
Note that $h$ is continuous and bounded. From (4.1.4) and applying the
dominated convergence theorem, we get the desired result.
\end{proof}

Clearly, the strict and relaxed optimal control problems have the same value
function.

\subsection{The maximum principle for nearly strict optimal controls}

In this subsection, we study near-optimal rather than optimal controls of
the control system. The precise definition of the near-optimality mainly
from [32], is

\begin{definition}
For a given $\varepsilon >0,$ an admissible pair $\left( X^{u^{\varepsilon
},\eta ^{\varepsilon }}\left( \cdot \right) ,u^{\varepsilon }\left( \cdot
\right) ,\eta ^{\varepsilon }\left( \cdot \right) \right) ,$ is called $%
\varepsilon $-optimal of system (2.1.1) if 
\begin{equation}
\left\vert J\left( u^{\varepsilon },\eta ^{\varepsilon }\right) -J\left(
u,\eta \right) \right\vert \leq \varepsilon .  \tag{4.2.1}
\end{equation}
\end{definition}

\begin{lemma}[{\textbf{Ekeland's principle [17]}}]
Let $\left( S,d\right) $ be a complete metric space and $\rho \left( \cdot
\right) :S\rightarrow R$ be lower-semicontinuous and bounded from below. For 
$\varepsilon \geq 0,$ suppose that $u^{\varepsilon }\in S$ satisfies%
\begin{equation*}
\rho \left( u^{\varepsilon }\right) \leq \inf\limits_{u\in S}\rho \left(
u\right) +\varepsilon .
\end{equation*}%
Then for any $\lambda >0,$ there exists $u^{\lambda }\in S$ such that%
\begin{equation}
\left\{ 
\begin{array}{lll}
\rho \left( u^{\lambda }\right) & \leq & \rho \left( u^{\varepsilon }\right)
, \\ 
d\left( u^{\lambda },u^{\varepsilon }\right) & \leq & \lambda , \\ 
\rho \left( u^{\lambda }\right) & \leq & \rho \left( u\right) +\frac{%
\varepsilon }{\lambda }d\left( u,u^{\lambda }\right) ,\text{ for all }u\in S.%
\end{array}%
\right.  \tag{4.2.2}
\end{equation}
\end{lemma}

\noindent For any $\left( u,\eta \right) ,$ $\left( v,\xi \right) \in 
\mathcal{U=}\mathcal{U}_{1}\times \mathcal{U}_{2},$ we define%
\begin{eqnarray*}
d_{1}\left( u\left( \cdot \right) ,v\left( \cdot \right) \right) &=&\tilde{P}%
\left\{ \left( t,\omega \right) \in \left[ 0,T\right] \times \Omega :u\left(
t,\omega \right) \neq v\left( t,\omega \right) \right\} , \\
d_{2}\left( \eta ,\xi \right) &=&\mathbb{E}\left( \sup\limits_{t\in \left[
0,T\right] }\left\vert \eta \left( t\right) -\xi \left( t\right) \right\vert
^{2}\right) ^{\frac{1}{2}}, \\
d\left( \left( u,\eta \right) ,\left( v,\xi \right) \right) &=&d_{1}\left(
u\left( \cdot \right) ,v\left( \cdot \right) \right) +d_{2}\left( \eta ,\xi
\right) .
\end{eqnarray*}%
where $\tilde{P}$ is the product measure of Lebesgue measure and $P.$ Since $%
U$ is closed, it can be shown that $\left( \mathcal{U},d\right) $ is a
complete metric space in [8] Lemma 4.5. Moreover, under the assumptions
(H1)-(H3), it is easy to check that $J(u\left( \cdot \right) ,\eta \left(
\cdot \right) )$ is continuous on $\mathcal{U}$ endowed with the metric $d$
above.

Now given any optimal relaxed control $\left( \mu \left( \cdot \right) ,\xi
\left( \cdot \right) \right) \in \mathcal{R}_{1}\times \mathcal{U}_{2}$, we
denote $X^{\mu ,\xi }\left( \cdot \right) $ the corresponding solution of
(4.1.2). From Lemma 11 and Lemma 12, there exists a sequence $\left(
u^{n}\left( \cdot \right) \right) _{n\geq 1}$ of strict control such that%
\begin{equation*}
\mu ^{n}\left( t\right) \left( \text{d}a\right) \text{d}t=\delta
_{u^{n}\left( t\right) }\left( \text{d}a\right) \text{d}t\rightarrow \mu
\left( t\right) \left( \text{d}a\right) \text{d}t,\text{ weakly, }P\text{%
-a.s.}
\end{equation*}%
and 
\begin{equation}
\lim\limits_{n\rightarrow +\infty }\mathbb{E}\left[ \sup\limits_{t\in \left[
0,T\right] }\left\vert X^{u^{n},\eta }\left( t\right) -X^{\mu ,\eta }\left(
t\right) \right\vert ^{2}\right] =0.  \tag{4.2.3}
\end{equation}%
From (4.1.5), there exists a positive sequence $\left( \varepsilon
_{n}\right) _{n\geq 1}$ with $\varepsilon _{n}\rightarrow 0,$ as $%
n\rightarrow +\infty $ such that%
\begin{equation}
J\left( u^{n}\left( \cdot \right) ,\eta \left( \cdot \right) \right)
=\inf\limits_{\left( v\left( \cdot \right) ,\xi \left( \cdot \right) \right)
\in \mathcal{U}}J\left( v\left( \cdot \right) ,\xi \left( \cdot \right)
\right) +\varepsilon _{n}.  \tag{4.2.4}
\end{equation}%
Then for $\lambda =\sqrt{\varepsilon _{n}},$ there exists $\left( u^{\sqrt{%
\varepsilon _{n}}}\left( \cdot \right) ,\eta \left( \cdot \right) \right)
\in \mathcal{U},$ such that 
\begin{equation}
\left\{ 
\begin{array}{l}
J\left( u^{\sqrt{\varepsilon _{n}}}\left( \cdot \right) ,\eta \left( \cdot
\right) \right) \leq \inf\limits_{\left( v,\xi \right) \in \mathcal{U}%
}J\left( v\left( \cdot \right) ,\xi \left( \cdot \right) \right)
+\varepsilon _{n}, \\ 
d\left( \left( u^{\sqrt{\varepsilon _{n}}}\left( \cdot \right) ,\eta \left(
\cdot \right) \right) ,\left( u^{n}\left( \cdot \right) ,\eta \left( \cdot
\right) \right) \right) \leq \sqrt{\varepsilon _{n}}, \\ 
J\left( \left( u^{\sqrt{\varepsilon _{n}}}\left( \cdot \right) ,\eta \left(
\cdot \right) \right) \right) \leq J\left( v\left( \cdot \right) ,\xi \left(
\cdot \right) \right) +\sqrt{\varepsilon _{n}}d\left[ \left( u^{\sqrt{%
\varepsilon _{n}}}\left( \cdot \right) ,\eta \left( \cdot \right) \right)
,\left( v\left( \cdot \right) ,\xi \left( \cdot \right) \right) \right] .%
\end{array}%
\right.   \tag{4.2.5}
\end{equation}%
Define 
\begin{equation}
\left( u^{\sqrt{\varepsilon _{n}},\alpha }\left( t\right) ,\eta \left(
t\right) \right) =\left\{ 
\begin{array}{ll}
\left( v,\eta \left( t\right) \right) , & \text{if }\tau \leq t\leq \tau
+\alpha , \\ 
\left( u^{\sqrt{\varepsilon _{n}},\alpha }\left( t\right) ,\eta \left(
t\right) \right) , & \text{otherwise,}%
\end{array}%
\right.   \tag{4.2.6}
\end{equation}%
and 
\begin{equation}
\left( u^{\sqrt{\varepsilon _{n}}}\left( t\right) ,\eta ^{\alpha }\left(
t\right) \right) =\left( u^{\sqrt{\varepsilon _{n}}}\left( t\right) ,\eta
\left( t\right) +\alpha \left( \xi \left( t\right) -\eta \left( t\right)
\right) \right) .  \tag{4.2.7}
\end{equation}%
Substituting (4.2.6) and (4.2.7) in (4.2.5), respectively, we have 
\begin{eqnarray*}
J\left( \left( u^{\sqrt{\varepsilon _{n}}}\left( \cdot \right) ,\eta \left(
\cdot \right) \right) \right)  &\leq &J\left( u^{\sqrt{\varepsilon _{n}}%
,\alpha }\left( \cdot \right) ,\eta \left( \cdot \right) \right)  \\
&&+\sqrt{\varepsilon _{n}}d\left[ \left( u^{\sqrt{\varepsilon _{n}}}\left(
\cdot \right) ,\eta \left( \cdot \right) \right) ,\left( u^{\sqrt{%
\varepsilon _{n}},\alpha }\left( \cdot \right) ,\eta \left( \cdot \right)
\right) \right] ,
\end{eqnarray*}%
and 
\begin{eqnarray*}
J\left( \left( u^{\sqrt{\varepsilon _{n}}}\left( \cdot \right) ,\eta \left(
\cdot \right) \right) \right)  &\leq &J\left( u^{\sqrt{\varepsilon _{n}}%
}\left( \cdot \right) ,\eta ^{\alpha }\left( \cdot \right) \right)  \\
&&+\sqrt{\varepsilon _{n}}d\left[ \left( u^{\sqrt{\varepsilon _{n}}}\left(
\cdot \right) ,\eta \left( \cdot \right) \right) ,u^{\sqrt{\varepsilon _{n}}%
}\left( \cdot \right) ,\eta ^{\alpha }\left( \cdot \right) \right] .
\end{eqnarray*}%
According to the definition of $d_{1}$ and $d_{2}$ and $M=\mathbb{E}\left[
\left\vert \eta \left( T\right) \right\vert ^{2}+\left\vert \xi \left(
T\right) \right\vert ^{2}\right] <+\infty ,$ we obtain that%
\begin{equation}
0\leq J\left( u^{\sqrt{\varepsilon _{n}},\alpha }\left( \cdot \right) ,\eta
\left( \cdot \right) \right) -J\left( \left( u^{\sqrt{\varepsilon _{n}}%
}\left( \cdot \right) ,\eta \left( \cdot \right) \right) \right) +\sqrt{%
\varepsilon _{n}}C_{1}\alpha ,  \tag{4.2.8}
\end{equation}%
and 
\begin{equation}
0\leq J\left( u^{\sqrt{\varepsilon _{n}}}\left( \cdot \right) ,\eta ^{\alpha
}\left( \cdot \right) \right) -J\left( \left( u^{\sqrt{\varepsilon _{n}}%
}\left( \cdot \right) ,\eta \left( \cdot \right) \right) \right) +\sqrt{%
\varepsilon _{n}}C_{2}\alpha ,  \tag{4.2.9}
\end{equation}%
where $C_{i},$ $i=1,2$ are positive constants depending on $U_{1}$, $M.$

As a consequence, we have the following theorem:

\begin{theorem}
Assume that (H1), (H3) and (H4) hold. For each $\varepsilon _{n}\in \left[
0,1\right] $, there exists a strict $\varepsilon _{n}$-optimal control $%
\left( u^{n}\left( \cdot \right) ,\eta \left( \cdot \right) \right) \in 
\mathcal{U}$ such that there exists a unique pair of adapted processes $%
\left( p^{n}\left( \cdot \right) ,q^{n}\left( \cdot \right) \right) $
satisfying 
\begin{equation*}
\mathbb{E}\left[ \sup\limits_{t\in \left[ 0,T\right] }\left\vert p^{n}\left(
t\right) \right\vert ^{2}\right] +\mathbb{E}\left[ \int_{0}^{T}\left\vert
q^{n}\left( t\right) \right\vert ^{2}\text{d}t\right] <+\infty ,
\end{equation*}%
which is the solution of the following mean-field BSDEs, 
\begin{equation}
\left\{ 
\begin{array}{l}
\text{d}p^{n}\left( t\right) =-\left[ b_{x}\left( t,X^{u^{n},\eta }\left(
t\right) ,\mathbb{E}\left[ X^{u^{n},\eta }\left( t\right) \right]
,u^{n}\left( t\right) \right) p^{n}\left( t\right) \right. \\ 
\qquad +\mathbb{E}\left[ b_{y}\left( t,X^{u^{n},\eta }\left( t\right) ,%
\mathbb{E}\left[ X^{u^{n},\eta }\left( t\right) \right] ,u^{n}\left(
t\right) \right) p^{n}\left( t\right) \right] \\ 
\qquad +\sigma _{x}\left( t,X^{u^{n},\eta }\left( t\right) ,\mathbb{E}\left[
X^{u^{n},\eta }\left( t\right) \right] \right) q^{n}\left( t\right) \\ 
\qquad +\mathbb{E}\left[ \sigma _{y}\left( t,X^{u^{n},\eta }\left( t\right) ,%
\mathbb{E}\left[ X^{u^{n},\eta }\left( t\right) \right] \right) q^{n}\left(
t\right) \right] \\ 
\qquad -f_{x}\left( t,X^{u^{n},\eta }\left( t\right) ,\mathbb{E}\left[
X^{u^{n},\eta }\left( t\right) \right] ,u^{n}\left( t\right) \right) \\ 
\qquad \left. -\mathbb{E}\left[ f_{y}\left( t,X^{u^{n},\eta }\left( t\right)
,\mathbb{E}\left[ X^{u^{n},\eta }\left( t\right) \right] ,u^{n}\left(
t\right) \right) \right] \right] \text{d}t+q^{n}\left( t\right) \text{d}%
W\left( t\right) , \\ 
p^{n}\left( T\right) =h_{x}\left( X^{u^{n},\eta }\left( T\right) ,\mathbb{E}%
\left[ X^{u^{n},\eta }\left( T\right) \right] \right) \\ 
\qquad +\mathbb{E}\left[ h_{y}\left( X^{u^{n},\eta }\left( T\right) ,\mathbb{%
E}\left[ X^{u^{n},\eta }\left( T\right) \right] \right) \right] .%
\end{array}%
\right.  \tag{4.2.10}
\end{equation}%
such that for all $\left( v\left( \cdot \right) ,\xi \left( \cdot \right)
\right) \in \mathcal{U}$ 
\begin{eqnarray*}
0 &\leq &\mathbb{E}\left[ \mathcal{H}\left( t,X^{u^{n},\eta }\left( t\right)
,v\left( t\right) ,p^{n}\left( t\right) ,q^{n}\left( t\right) \right) \right.
\\
&&\left. -\mathcal{H}\left( t,X^{u^{n},\eta }\left( t\right) ,u^{n}\left(
t\right) ,p^{n}\left( t\right) ,q^{n}\left( t\right) \right) \right] +\sqrt{%
\varepsilon _{n}}C_{1}\alpha ,
\end{eqnarray*}%
\begin{equation}
\tag{4.2.11}
\end{equation}%
and%
\begin{equation}
0\leq \mathbb{E}\left[ \int_{0}^{T}\left( \varphi \left( t\right) +G^{\ast
}\left( t\right) p^{n}\left( t\right) \right) \text{d}\left( \xi \left(
t\right) -\eta \left( t\right) \right) \right] +\sqrt{\varepsilon _{n}}%
C_{2}\alpha ,  \tag{4.2.12}
\end{equation}%
where $C_{i},$ $i=1,2$ are positive constants.
\end{theorem}

\begin{proof}
From (4.2.8) and (4.2.9), using the same method as in [8, Theorem 3.6,
Theorem 4.6] , we obtain (4.2.10) and (4.2.11), respectively.
\end{proof}

\subsection{Necessary Optimality Conditions for Relaxed Singular Controls}

We have

\begin{theorem}[\textbf{Relaxed maximum principle in integral form}]
Assume that (H1), (H3) and (H4) hold. Let $\left( \mu \left( \cdot \right)
,\eta \left( \cdot \right) \right) $ be an optimal relaxed control
minimizing the cost $J$ over $\mathcal{R}_{1}\times \mathcal{U}_{2},$ and
let $X^{\mu ,\eta }\left( \cdot \right) $ be the corresponding optimal
trajectory. Then there exists a unique pair of adapted processes $\left(
p^{\mu ,\eta }\left( \cdot \right) ,q^{\mu ,\eta }\left( \cdot \right)
\right) $ 
\begin{equation*}
\mathbb{E}\left[ \sup\limits_{t\in \left[ 0,T\right] }\left\vert p^{\mu
,\eta }\left( t\right) \right\vert ^{2}\right] +\mathbb{E}\left[
\int_{0}^{T}\left\vert q^{\mu ,\eta }\left( t\right) \right\vert ^{2}\text{d}%
t\right] <+\infty ,
\end{equation*}%
which is the solution of the following mean-field BSDEs%
\begin{equation}
\left\{ 
\begin{array}{l}
\text{d}p^{\mu ,\eta }\left( t\right) =-\int_{U}b_{x}\left( t,X^{\mu ,\eta
}\left( t\right) ,\mathbb{E}\left[ X^{\mu ,\eta }\left( t\right) \right]
,a\right) \mu \left( t,\text{d}a\right) p^{\mu ,\eta }\left( t\right) \text{d%
}t \\ 
\qquad \qquad -\sigma _{x}\left( t,X^{\mu ,\eta }\left( t\right) ,\mathbb{E}%
\left[ X^{\mu ,\eta }\left( t\right) \right] \right) q^{\mu ,\eta }\left(
t\right) \text{d}t \\ 
\qquad \qquad -\int_{U}f_{x}\left( t,X^{\mu ,\eta }\left( t\right) ,\mathbb{E%
}\left[ X^{\mu ,\eta }\left( t\right) \right] ,a\right) \mu \left( t,\text{d}%
a\right) \text{d}t \\ 
\qquad \qquad -\mathbb{E}\left[ \int_{U}b_{y}\left( t,X^{\mu ,\eta }\left(
t\right) ,\mathbb{E}\left[ X^{\mu ,\eta }\left( t\right) \right] ,a\right)
\mu \left( t,\text{d}a\right) p^{\mu ,\eta }\left( t\right) \text{d}t\right]
\\ 
\qquad \qquad -\mathbb{E}\left[ \sigma _{y}\left( t,X^{\mu ,\eta }\left(
t\right) ,\mathbb{E}\left[ X^{\mu ,\eta }\left( t\right) \right] \right)
q^{\mu ,\eta }\left( t\right) \right] \text{d}t \\ 
\qquad \qquad -\mathbb{E}\left[ \int_{U}f_{y}\left( t,X^{\mu ,\eta }\left(
t\right) ,\mathbb{E}\left[ X^{\mu ,\eta }\left( t\right) \right] ,a\right)
\mu \left( t,\text{d}a\right) \text{d}t\right] +q^{\mu ,\eta }\left(
t\right) dW\left( t\right) \\ 
p^{\mu ,\eta }\left( T\right) =h_{x}\left( X^{\mu ,\eta }\left( T\right) ,%
\mathbb{E}\left[ X^{\mu ,\eta }\left( T\right) \right] \right) +\mathbb{E}%
\left[ h_{y}\left( X^{\mu ,\eta }\left( T\right) ,\mathbb{E}\left[ X^{\mu
,\eta }\left( T\right) \right] \right) \right] ,%
\end{array}%
\right.  \tag{4.3.1}
\end{equation}%
such that for all $\left( v\left( \cdot \right) ,\xi \left( \cdot \right)
\right) \in \mathcal{U}_{1}\times \mathcal{U}_{2}$, we have 
\begin{eqnarray*}
&&\mathbb{E}\left[ H\left( t,X^{\mu ,\eta }\left( t\right) ,v\left( t\right)
,\eta \left( t\right) ,p^{\mu ,\eta }\left( t\right) ,q^{\mu ,\eta }\left(
t\right) \right) \right. \\
&&-\left. H\left( t,X^{\mu ,\eta }\left( t\right) ,\mu \left( t\right) ,\eta
\left( t\right) ,p^{\mu ,\eta }\left( s\right) ,q^{\mu ,\eta }\left(
t\right) \right) \right] \\
&\geq &0.
\end{eqnarray*}%
\begin{equation}
\tag{4.3.2}
\end{equation}%
\begin{equation}
0\leq \mathbb{E}\left[ \int_{0}^{T}\left( \varphi \left( s\right) +G^{\ast
}\left( s\right) p^{\mu ,\eta }\left( s\right) \right) \text{d}\left( \eta
\left( s\right) -\xi \left( s\right) \right) \right] ,  \tag{4.3.3}
\end{equation}%
where 
\begin{eqnarray*}
&&H\left( t,X^{\mu ,\eta }\left( t\right) ,\mu \left( t\right) ,\eta \left(
t\right) ,p^{\mu ,\eta }\left( s\right) ,q^{\mu ,\eta }\left( t\right)
\right) \\
&=&\int_{U}H\left( t,X^{\mu ,\eta }\left( t\right) ,a,\eta \left( t\right)
,p^{\mu ,\eta }\left( s\right) ,q^{\mu ,\eta }\left( t\right) \right) \mu
\left( t,\text{d}a\right) .
\end{eqnarray*}
\end{theorem}

To prove Theorem 2, we need the following lemma.

\begin{lemma}
Let $\left( p^{n}\left( \cdot \right) ,q^{n}\left( \cdot \right) \right) $
and $\left( p^{\mu ,\eta }\left( \cdot \right) ,q^{\mu ,\eta }\left( \cdot
\right) \right) $ be the solutions of (4.2.10) and (4.3.1), respectively.
Then we have%
\begin{equation}
\lim\limits_{n\rightarrow \infty }\left( \mathbb{E}\left[ \sup\limits_{t\in %
\left[ 0,T\right] }\left\vert p^{\mu ,\eta }\left( t\right) -p^{n}\left(
t\right) \right\vert ^{2}\right] +\mathbb{E}\left[ \int_{0}^{T}\left\vert
q^{\mu ,\eta }\left( t\right) -q^{n}\left( t\right) \right\vert ^{2}\text{d}t%
\right] \right) =0.  \tag{4.3.4}
\end{equation}
\end{lemma}

\begin{proof}
Set%
\begin{equation*}
\left\{ 
\begin{array}{l}
\mu ^{n}\left( t,\text{d}a\right) =\delta _{u^{n}\left( t\right) }\left( 
\text{d}a\right) , \\ 
b^{1,\mu }\left( t\right) =\int_{U}b_{x}\left( t,X^{\mu ,\eta }\left(
t\right) ,\mathbb{E}\left[ X^{\mu ,\eta }\left( t\right) \right] ,a\right)
\mu \left( t,\text{d}a\right) , \\ 
b^{2,\mu }\left( t\right) =\int_{U}b_{y}\left( t,X^{\mu ,\eta }\left(
t\right) ,\mathbb{E}\left[ X^{\mu ,\eta }\left( t\right) \right] ,a\right)
\mu \left( t,\text{d}a\right) , \\ 
b^{1,n}\left( t\right) =\int_{U}b_{x}\left( t,X^{u^{n},\eta }\left( t\right)
,\mathbb{E}\left[ X^{u^{n},\eta }\left( t\right) \right] ,a\right) \mu
^{n}\left( t,\text{d}a\right) , \\ 
b^{2,n}\left( t\right) =\int_{U}b_{y}\left( t,X^{u^{n},\eta }\left( t\right)
,\mathbb{E}\left[ X^{u^{n},\eta }\left( t\right) \right] ,a\right) \mu
^{n}\left( t,\text{d}a\right) , \\ 
\sigma ^{1,\mu }\left( t\right) =\sigma _{x}\left( t,X^{\mu ,\eta }\left(
t\right) ,\mathbb{E}\left[ X^{\mu ,\eta }\left( t\right) \right] \right) ,
\\ 
\sigma ^{2,\mu }\left( t\right) =\sigma _{y}\left( t,X^{\mu ,\eta }\left(
t\right) ,\mathbb{E}\left[ X^{\mu ,\eta }\left( t\right) \right] \right) ,
\\ 
\sigma ^{1,n}\left( t\right) =\sigma _{x}\left( t,X^{u^{n},\eta }\left(
t\right) ,\mathbb{E}\left[ X^{u^{n},\eta }\left( t\right) \right] \right) ,
\\ 
\sigma ^{2,n}\left( t\right) =\sigma _{y}\left( t,X^{u^{n},\eta }\left(
t\right) ,\mathbb{E}\left[ X^{u^{n},\eta }\left( t\right) \right] \right) ,
\\ 
f^{1,\mu }\left( t\right) =\int_{U}f_{x}\left( t,X^{\mu ,\eta }\left(
t\right) ,\mathbb{E}\left[ X^{\mu ,\eta }\left( t\right) \right] ,a\right)
\mu \left( t,\text{d}a\right) , \\ 
f^{2,\mu }\left( t\right) =\int_{U}f_{y}\left( t,X^{\mu ,\eta }\left(
t\right) ,\mathbb{E}\left[ X^{\mu ,\eta }\left( t\right) \right] ,a\right)
\mu \left( t,\text{d}a\right) , \\ 
f^{1,n}\left( t\right) =\int_{U}f_{x}\left( t,X^{u^{n},\eta }\left( t\right)
,\mathbb{E}\left[ X^{u^{n},\eta }\left( t\right) \right] ,a\right) \mu
^{n}\left( t,\text{d}a\right) , \\ 
f^{2,n}\left( t\right) =\int_{U}f_{y}\left( t,X^{u^{n},\eta }\left( t\right)
,\mathbb{E}\left[ X^{u^{n},\eta }\left( t\right) \right] ,a\right) \mu
^{n}\left( t,\text{d}a\right) ,%
\end{array}%
\right.
\end{equation*}%
Since $l_{x},l_{y},$ $l=b,\sigma ,h$ are bounded and continuous, from Lemma
11 and Lemma 12, it is easy to get%
\begin{equation}
\left\{ 
\begin{array}{c}
\lim\limits_{n\rightarrow \infty }\mathbb{E}\left[ \left\vert l^{1,\mu
}\left( t\right) -l^{1,n}\left( t\right) \right\vert ^{2}\right] =0, \\ 
\lim\limits_{n\rightarrow \infty }\mathbb{E}\left[ \left\vert l^{2,\mu
}\left( t\right) -l^{2,n}\left( t\right) \right\vert ^{2}\right] =0,%
\end{array}%
\right.  \tag{4.3.5}
\end{equation}%
where $l$ stands for $b,\sigma ,h,$ respectively.

To get (4.3.4), applying It\^{o}'s formula to $\left( p^{\mu ,\eta }\left(
t\right) -p^{n}\left( t\right) \right) ^{2}$ on $\left[ t,T\right] ,$ we have%
\begin{eqnarray*}
&&\mathbb{E}\left[ \left\vert p^{\mu ,\eta }\left( t\right) -p^{n}\left(
t\right) \right\vert ^{2}\right] +\mathbb{E}\int_{t}^{T}\left\vert q^{\mu
,\eta }\left( s\right) -q^{n}\left( s\right) \right\vert ^{2}\text{d}s \\
&=&\mathbb{E}\left[ \left\vert h_{x}\left( X^{\mu ,\eta }\left( T\right) ,%
\mathbb{E}\left[ X^{\mu ,\eta }\left( T\right) \right] \right) +\mathbb{E}%
\left[ h_{y}\left( X^{\mu ,\eta }\left( T\right) ,\mathbb{E}\left[ X^{\mu
,\eta }\left( T\right) \right] \right) \right] \right. \right.  \\
&&\left. \left. -h_{x}\left( X^{u^{n},\eta }\left( T\right) ,\mathbb{E}\left[
X^{u,\eta }\left( T\right) \right] \right) -\mathbb{E}\left[ h_{y}\left(
X^{u^{n},\eta }\left( T\right) ,\mathbb{E}\left[ X^{u^{n},\eta }\left(
T\right) \right] \right) \right] \right\vert ^{2}\right]  \\
&&+2\mathbb{E}\int_{t}^{T}\left( p^{\mu ,\eta }\left( s\right) -p^{n}\left(
s\right) \right) \left( \Pi ^{\mu ,\eta }\left( s\right) -\Pi ^{n}\left(
s\right) \right) \text{d}s,
\end{eqnarray*}%
where%
\begin{equation}
\left\{ 
\begin{array}{l}
\Pi ^{\mu ,\eta }\left( s\right) =b^{1,\mu }\left( s\right) p^{\mu ,\eta
}\left( s\right) +b^{2,\mu }\left( s\right) \mathbb{E}\left[ p^{\mu ,\eta
}\left( s\right) \right] +\sigma ^{1,\mu }\left( s\right) q^{\mu ,\eta
}\left( s\right)  \\ 
\qquad +\sigma ^{2,\mu }\left( s\right) \mathbb{E}\left[ q^{\mu ,\eta
}\left( s\right) \right] +h^{1,\mu }\left( t\right) +h^{2,\mu }\left(
t\right) , \\ 
\Pi ^{n}\left( s\right) =b^{1,n}\left( s\right) p^{n}\left( s\right)
+b^{2,n}\left( s\right) \mathbb{E}\left[ p^{n}\left( s\right) \right]
+\sigma ^{1,n}\left( s\right) q^{n}\left( s\right) +\sigma ^{2,n}\left(
s\right) \mathbb{E}\left[ q^{n}\left( s\right) \right]  \\ 
\qquad +h^{1,n}\left( s\right) +h^{2,n}\left( s\right) .%
\end{array}%
\right.   \tag{4.3.6}
\end{equation}%
Using the inequality $ab\leq \frac{\varepsilon }{2}a^{2}+\frac{1}{%
2\varepsilon }b^{2},$ we obtain%
\begin{eqnarray*}
&&\mathbb{E}\left[ \left\vert p^{\mu ,\eta }\left( t\right) -p^{n}\left(
t\right) \right\vert ^{2}\right] +\mathbb{E}\int_{t}^{T}\left\vert q^{\mu
,\eta }\left( s\right) -q^{n}\left( s\right) \right\vert ^{2}\text{d}s \\
&\leq &\mathbb{E}\left[ \left\vert h_{x}\left( X^{\mu ,\eta }\left( T\right)
,\mathbb{E}\left[ X^{\mu ,\eta }\left( T\right) \right] \right) +\mathbb{E}%
\left[ h_{y}\left( X^{\mu ,\eta }\left( T\right) ,\mathbb{E}\left[ X^{\mu
,\eta }\left( T\right) \right] \right) \right] \right. \right.  \\
&&\left. \left. -h_{x}\left( X^{u^{n},\eta }\left( T\right) ,\mathbb{E}\left[
X^{u,\eta }\left( T\right) \right] \right) -h_{y}\left( X^{u^{n},\eta
}\left( T\right) ,\mathbb{E}\left[ X^{u^{n},\eta }\left( T\right) \right]
\right) \right\vert ^{2}\right]  \\
&&+\frac{1}{\varepsilon }\mathbb{E}\int_{t}^{T}\left\vert p^{\mu ,\eta
}\left( s\right) -p^{n}\left( s\right) \right\vert ^{2}\text{d}s+\varepsilon 
\mathbb{E}\int_{t}^{T}\left\vert \Pi ^{\mu ,\eta }\left( s\right) -\Pi
^{n}\left( s\right) \right\vert ^{2}\text{d}s \\
&\leq &\left( \frac{1}{\varepsilon }+24M\varepsilon \right) \mathbb{E}%
\int_{t}^{T}\left\vert p^{\mu ,\eta }\left( s\right) -p^{n}\left( s\right)
\right\vert ^{2}\text{d}s \\
&&+24M\varepsilon \mathbb{E}\int_{t}^{T}\left\vert q^{\mu ,\eta }\left(
s\right) -q^{n}\left( s\right) \right\vert ^{2}\text{d}s+\varepsilon \Theta
^{n}\left( t\right) ,
\end{eqnarray*}%
where%
\begin{eqnarray*}
\Theta ^{n}\left( t\right)  &=&\left\{ \frac{1}{\varepsilon }\mathbb{E}\left[
\left\vert h_{x}\left( X^{\mu ,\eta }\left( T\right) ,\mathbb{E}\left[
X^{\mu ,\eta }\left( T\right) \right] \right) +\mathbb{E}\left[ h_{y}\left(
X^{\mu ,\eta }\left( T\right) ,\mathbb{E}\left[ X^{\mu ,\eta }\left(
T\right) \right] \right) \right] \right. \right. \right.  \\
&&\left. \left. \left. -h_{x}\left( X^{u^{n},\eta }\left( T\right) ,\mathbb{E%
}\left[ X^{u,\eta }\left( T\right) \right] \right) -h_{y}\left(
X^{u^{n},\eta }\left( T\right) ,\mathbb{E}\left[ X^{u^{n},\eta }\left(
T\right) \right] \right) \right\vert ^{2}\right] \right\}  \\
&&+12\mathbb{E}\int_{t}^{T}\left\vert \left( b^{1,\mu }\left( s\right)
-b^{1,n}\left( s\right) \right) p^{n}\left( s\right) \right\vert ^{2}\text{d}%
s \\
&&+12\mathbb{E}\int_{t}^{T}\left\vert \left( b^{2,\mu }\left( s\right)
-b^{2,n}\left( s\right) \right) p^{n}\left( s\right) \right\vert ^{2}\text{d}%
s \\
&&+12\mathbb{E}\int_{t}^{T}\left\vert \left( \sigma ^{1,\mu }\left( s\right)
-\sigma ^{1,n}\left( s\right) \right) q^{n}\left( s\right) \right\vert ^{2}%
\text{d}s \\
&&+12\mathbb{E}\int_{t}^{T}\left\vert \left( \sigma ^{2,\mu }\left( s\right)
-\sigma ^{2,n}\left( s\right) \right) q^{n}\left( s\right) \right\vert ^{2}%
\text{d}s \\
&&+6\left\vert h_{x}\left( X^{\mu ,\eta }\left( T\right) ,\mathbb{E}\left[
X^{\mu ,\eta }\left( T\right) \right] \right) --h_{x}\left( X^{u^{n},\eta
}\left( T\right) ,\mathbb{E}\left[ X^{u,\eta }\left( T\right) \right]
\right) \right\vert ^{2} \\
&&+6\left\vert \mathbb{E}h\left[ _{y}\left( X^{\mu ,\eta }\left( T\right) ,%
\mathbb{E}\left[ X^{\mu ,\eta }\left( T\right) \right] \right) -h_{y}\left(
X^{u^{n},\eta }\left( T\right) ,\mathbb{E}\left[ X^{u^{n},\eta }\left(
T\right) \right] \right) \right] \right\vert ^{2}.
\end{eqnarray*}%
Picking $\varepsilon =\frac{1}{48M},$ we have%
\begin{eqnarray*}
&&\mathbb{E}\left[ \left\vert p^{\mu ,\eta }\left( t\right) -p^{n}\left(
t\right) \right\vert ^{2}\right] +\frac{1}{2}\mathbb{E}\int_{t}^{T}\left%
\vert q^{\mu ,\eta }\left( s\right) -q^{n}\left( s\right) \right\vert ^{2}%
\text{d}s \\
&\leq &C\mathbb{E}\int_{t}^{T}\left\vert p^{\mu ,\eta }\left( s\right)
-p^{n}\left( s\right) \right\vert ^{2}\text{d}s+C\Theta ^{n}\left( t\right) ,
\end{eqnarray*}%
where $C>0$ depends on $M.$

We are going to show that 
\begin{equation}
\lim\limits_{n\rightarrow +\infty }\Theta ^{n}\left( t\right) =0. 
\tag{4.3.7}
\end{equation}%
Using Cauchy-Schwarz inequality we obtain%
\begin{eqnarray*}
&&\mathbb{E}\int_{t}^{T}\left\vert \left( b^{1,\mu }\left( s\right)
-b^{1,n}\left( s\right) \right) p^{n}\left( s\right) \right\vert \text{d}s \\
&\leq &\int_{t}^{T}\left( \mathbb{E}\left\vert b^{1,\mu }\left( s\right)
-b^{1,n}\left( s\right) \right\vert ^{2}\right) ^{\frac{1}{2}}\left( \mathbb{%
E}\left\vert p^{n}\left( s\right) \right\vert ^{2}\right) ^{\frac{1}{2}}.
\end{eqnarray*}%
From (4.3.5), it follows that 
\begin{equation*}
\mathbb{E}\int_{t}^{T}\left\vert \left( b^{1,\mu }\left( s\right)
-b^{1,n}\left( s\right) \right) p^{n}\left( s\right) \right\vert \text{d}%
s\rightarrow 0,\text{ as }n\rightarrow \infty .
\end{equation*}%
On the other hand, by (H1), It is easy to see that $\mathbb{E}\left\vert
p^{n}\left( s\right) \right\vert ^{2}<+\infty ,$ uniformly$.$ Then we have 
\begin{eqnarray*}
&&\mathbb{E}\int_{t}^{T}\left\vert \left( b^{1,\mu }\left( s\right)
-b^{1,n}\left( s\right) \right) p^{n}\left( s\right) \right\vert ^{2}\text{d}%
s \\
&\leq &M\mathbb{E}\left[ \sup\limits_{t\in \left[ 0,T\right] }\left\vert
p^{n}\left( s\right) \right\vert \int_{t}^{T}\left\vert \left( b^{1,\mu
}\left( s\right) -b^{1,n}\left( s\right) \right) p^{n}\left( s\right)
\right\vert \text{d}s\right] \rightarrow 0,\text{ as }n\rightarrow +\infty ,
\end{eqnarray*}%
\begin{equation}
\tag{4.3.8}
\end{equation}%
where $M$ depends Lipschitz constant of $b_{x}.$ We can also get 
\begin{equation}
\left\{ 
\begin{array}{c}
\lim\limits_{n\rightarrow +\infty }\mathbb{E}\int_{t}^{T}\left\vert \left(
b^{2,\mu }\left( s\right) -b^{2,n}\left( s\right) \right) p^{n}\left(
s\right) \right\vert ^{2}\text{d}s=0, \\ 
\lim\limits_{n\rightarrow +\infty }\mathbb{E}\int_{t}^{T}\left\vert \left(
\sigma ^{1,\mu }\left( s\right) -\sigma ^{1,n}\left( s\right) \right)
q^{n}\left( s\right) \right\vert ^{2}\text{d}s=0, \\ 
\lim\limits_{n\rightarrow +\infty }\mathbb{E}\int_{t}^{T}\left\vert \left(
\sigma ^{2,\mu }\left( s\right) -\sigma ^{2,n}\left( s\right) \right)
q^{n}\left( s\right) \right\vert ^{2}\text{d}s=0,%
\end{array}%
\right.   \tag{4.3.9}
\end{equation}%
At last, since $h_{x},h_{y}$ are continuous and bounded, we have%
\begin{eqnarray*}
&&\lim\limits_{n\rightarrow +\infty }\mathbb{E}\left[ \left\vert h_{x}\left(
X^{\mu ,\eta }\left( T\right) ,\mathbb{E}\left[ X^{\mu ,\eta }\left(
T\right) \right] \right) +\mathbb{E}\left[ h_{y}\left( X^{\mu ,\eta }\left(
T\right) ,\mathbb{E}\left[ X^{\mu ,\eta }\left( T\right) \right] \right) %
\right] \right. \right.  \\
&&\left. \left. -h_{x}\left( X^{u^{n},\eta }\left( T\right) ,\mathbb{E}\left[
X^{u,\eta }\left( T\right) \right] \right) -\mathbb{E}\left[ h_{y}\left(
X^{u^{n},\eta }\left( T\right) ,\mathbb{E}\left[ X^{u^{n},\eta }\left(
T\right) \right] \right) \right] \right\vert ^{2}\right]  \\
&=&0.
\end{eqnarray*}%
\begin{equation}
\tag{4.3.10}
\end{equation}%
From (4.3.8)-(4.3.10) we claim that (4.3.7) holds. Applying Gronwall's
inequality, we get the desired result (4,3,4).
\end{proof}

\begin{proof}
Proof of Theorem 2. Suppose that $\left( \mu \left( \cdot \right) ,\eta
\left( \cdot \right) \right) $ is the optimal relaxed control. Then from
Theorem 1, we know also that there exists a sequence $\left( u^{n}\left(
\cdot \right) ,\eta \left( \cdot \right) \right) _{n\geq 1}$ converge to the
relaxed counterpart as $n\rightarrow +\infty ,$ such that (4.2.11), (4.2.12)
hold for all $\left( v\left( \cdot \right) ,\xi \left( \cdot \right) \right) 
$ in $\mathcal{U}_{1}\times \mathcal{U}_{2}.$ Letting $n$ tend to infinite
and using Lemma 9, we get the desired result.
\end{proof}

\begin{theorem}
Assume that (H1), (H3) and (H4) hold. Let $\left( \mu \left( \cdot \right)
,\eta \left( \cdot \right) \right) $ be an optimal relaxed control
minimizing the cost $J$ over $\mathcal{R}_{1}\times \mathcal{U}_{2},$ and
let $X^{\mu ,\eta }\left( \cdot \right) $ be the corresponding optimal
trajectory. Then there exists a unique pair of adapted processes $\left(
p^{\mu ,\eta }\left( \cdot \right) ,q^{\mu ,\eta }\left( \cdot \right)
\right) $ of BSDE (4.3.1), such that for all $\left( v\left( \cdot \right)
,\xi \left( \cdot \right) \right) \in \mathcal{U}_{1}\times \mathcal{U}_{2}$%
, we have 
\begin{eqnarray*}
&&H\left( t,X^{\mu ,\eta }\left( t\right) ,\mu \left( t\right) ,\eta \left(
t\right) ,p^{\mu ,\eta }\left( s\right) ,q^{\mu ,\eta }\left( t\right)
\right)  \\
&=&\min\limits_{v\in U_{1}}H\left( t,X^{\mu ,\eta }\left( t\right) ,v,\eta
\left( t\right) ,p^{\mu ,\eta }\left( t\right) ,q^{\mu ,\eta }\left(
t\right) \right) ,\text{ }P\text{-a.s., d}t\text{-a.e.}
\end{eqnarray*}%
\begin{equation}
\tag{4.3.11}
\end{equation}%
\begin{equation}
P\left\{ \varphi _{i}\left( t\right) +G_{i}\left( t\right) p^{\mu ,\eta
}\left( t\right) \geq 0\right\} =1,  \tag{4.3.12}
\end{equation}%
\begin{equation}
P\left\{ \sum_{i=1}^{m}\mathbf{I}_{\varphi _{i}\left( t\right) +G_{i}\left(
t\right) p^{\mu ,\eta }\left( t\right) d\eta _{i}\left( t\right) \geq
0}=0\right\} =1.  \tag{4.3.13}
\end{equation}
\end{theorem}

\begin{proof}
(4.3.11) can be derived from (4.3.2). The assertions (4.3.12) and (4.3.13)
are proved exactly as in Theorem 3.7 in [8].
\end{proof}

\begin{corollary}
Under the same assumptions in Theorem 4, we have%
\begin{eqnarray*}
&&H\left( t,X^{\mu ,\eta }\left( t\right) ,\mu \left( t\right) ,\eta \left(
t\right) ,p^{\mu ,\eta }\left( s\right) ,q^{\mu ,\eta }\left( t\right)
\right)  \\
&=&\min\limits_{\varsigma \in P\left( U_{1}\right) }H\left( t,X^{\mu ,\eta
}\left( t\right) ,\varsigma ,\eta \left( t\right) ,p^{\mu ,\eta }\left(
t\right) ,q^{\mu ,\eta }\left( t\right) \right) ,\text{ }P\text{-a.s., d}t%
\text{-a.e.}
\end{eqnarray*}%
\begin{equation}
\tag{4.3.14}
\end{equation}%
\begin{equation*}
P\left\{ \varphi _{i}\left( t\right) +G_{i}\left( t\right) p^{\mu ,\eta
}\left( t\right) \geq 0\right\} =1,
\end{equation*}%
\begin{equation*}
P\left\{ \sum_{i=1}^{m}\mathbf{I}_{\varphi _{i}\left( t\right) +G_{i}\left(
t\right) p^{\mu ,\eta }\left( t\right) d\eta _{i}\left( t\right) \geq
0}=0\right\} =1.
\end{equation*}
\end{corollary}

\begin{proof}
(4.3.14) can be proved the same as Corollary 4.8 in [8].
\end{proof}

\begin{remark}
Taking $\mu \left( t,\text{d}a\right) =\delta _{u\left( t\right) }\left( 
\text{d}a\right) $, we recover Theorem 1.
\end{remark}

\begin{remark}
As you have observed that, in our paper, the control variable does not enter
the diffusion term. In fact, for the classical case, that is, both drift and
diffusion terms containing control variables, the similar optimal control
problem has been studied by Andersson, in [2]. As for mean-field case, we
will investigate it in our future work.
\end{remark}

\end{document}